\documentclass[a4paper,english]{article}

\usepackage[hyphens]{url} 	%
\usepackage{geometry} 		%
\usepackage[T1]{fontenc}	%
\usepackage[utf8]{inputenc} %
\usepackage{lmodern}		%
\usepackage[english]{babel} %
\usepackage{graphicx}		%
\usepackage[algoruled,english]{algorithm2e}
\SetKwInput{KwInput}{Input}                %
\SetKwInput{KwOutput}{Output}              %

\usepackage{pgfplots}
\usepackage{pgfplotstable}
\pgfplotsset{compat=1.18}
\usepackage{booktabs}
\usepackage{multirow}		%

\usepackage{array} 			%
\usepackage[hypcap=true]{caption}
\usepackage{subcaption} 	%
\usepackage[dvipsnames]{xcolor}		%
\usepackage{float} 		%
\usepackage{flafter}     %
\usepackage[section]{placeins} %
\usepackage{mathtools}		%
\usepackage{amssymb}		%
\usepackage{nicefrac}		%
\usepackage[thmmarks, amsmath, amsthm]{ntheorem} 	%
\usepackage{bm}			%

\usepackage{titling} 		%

\usepackage{fancyhdr} 		%
\usepackage{titlesec} 		%
\titleformat*{\section}{\large\bfseries}
\usepackage{enumitem} 		%

\usepackage{csquotes}		%
\usepackage[style=numeric,backend=biber,giveninits=true,alldates=year,doi=false,url=false,isbn=false]{biblatex}
\DeclareNameAlias{default}{given-family}
\DeclareNameAlias{sortname}{given-family}
\AtEveryBibitem{%
  \ifboolexpr{
    test {\ifentrytype{book}} or
    test {\ifentrytype{incollection}} or
    test {\ifentrytype{inproceedings}}
  }{%
    \clearlist{location}%
    \clearname{editor}%
    \iffieldundef{series}{}{%
      \clearfield{series}%
      \clearfield{volume}%
      \clearfield{number}%
    }%
  }{}%
  \ifentrytype{book}{\clearfield{pagetotal}}{}%
}
\usepackage[hidelinks, pdfstartview=FitB]{hyperref}

\DeclareRobustCommand{\IC}{\mathbb C}

\newcommand{\IN}{\mathbb N}

\newcommand{\IR}{\mathbb R}

\newcommand{\N}{\mathbb N}

\newcommand{\R}{\mathbb R}
\newcommand{\T}{\mathbb T}

\newcommand{\sC}{{\mathcal C}}

\newcommand{\sG}{{\mathcal G}}

\newcommand{\sM}{{\mathcal M}}

\newcommand{\sO}{{\mathcal O}}
\newcommand{\sP}{{\Pi}}

\newcommand{\sX}{{\mathcal X}}

\newcommand{\vc}[1]{\bm{#1}}
\newcommand{\mt}[1]{\bm{#1}}

\newcommand{\cond}{\mathrm{cond}}

\newcommand{\supp}{\operatorname{supp}}
\DeclareMathOperator*{\argmin}{\arg\!\min}

\newcommand{\SO}{\mathrm{SO}}
\renewcommand{\S}[1]{\mathcal{S}^{#1}}
\DeclareMathOperator{\Harm}{Harm}
\NewDocumentCommand{\bandlimFunc}{gg}{%
  \IfNoValueTF{#2}
  {\sP_{#1}(\SO(3))}
  {\sP_{#1}(#2)}
}
\renewcommand{\L}[1]{\mathrm{L}^{#1}}
\DeclarePairedDelimiter{\abs}{|}{|}
\DeclarePairedDelimiter{\braces}{(}{)}
\DeclarePairedDelimiter{\norm}{\lVert}{\rVert}
\DeclarePairedDelimiterX{\scp}[2]{\langle}{\rangle}{\,#1\,,\,#2\,}
\DeclarePairedDelimiterX{\set}[2]{\{}{\}}{
		\ifthenelse {\equal{#2}{} }
				{ \def\temp{}   }													%
				{ \def\temp{\,\delimsize|\, #2} }					%
		 #1 \temp
}												 	%

\newcommand{\complexity}[1]{\mathcal O (#1)}

\newcommand{\ie}{i.e.\ }

\newcommand{\iid}{i.i.d.\ }

\makeatletter %
\let\@@pmod\pmod
\DeclareRobustCommand{\pmod}{\@ifstar\@pmods\@@pmod}
\def\@pmods#1{\mkern4mu({\operator@font mod}\mkern 6mu#1)}
\makeatother

\newcolumntype{C}[1]{>{\centering\arraybackslash}p{#1}}

\newtheorem{theorem}{Theorem}[section]

\newtheorem{remark}[theorem]{Remark}

\usepackage{orcidlink}
\usepackage[nameinlink]{cleveref} %

\hypersetup{
  pdftitle={Fast and Stable Harmonic Approximation from Ill-distributed Data using Moving Least Squares},
  pdfauthor={Ralf Hielscher, Tim Pöschl, Erik Wünsche},
  pdfsubject={Harmonic approximation via moving least squares},
  pdfkeywords={scattered data approximation, moving least squares, harmonic expansion, spherical harmonics, fast Fourier transform}
}

\fussy						%
\usepackage[final]{microtype}	%

\begin{document}

\title{Fast and Stable Harmonic Approximation \\
  from Ill-distributed Data using Moving Least Squares}       

\author{Ralf
  Hielscher\orcidlink{0000-0002-6342-1799}%
  \thanks{ralf.hielscher@math.tu-freiberg.de}~,~%
  Tim Pöschl\orcidlink{0009-0006-8313-9850}%
  \thanks{tim.poeschl@math.tu-freiberg.de}~,~%
  Erik Wünsche\orcidlink{0000-0001-5178-6654}%
  \thanks{erik.wuensche@math.tu-freiberg.de}}

\date{Institute for Applied Analysis, TU Bergakademie Freiberg \\
  September 30, 2026}

\maketitle

\begin{abstract}
  We consider harmonic approximation on the torus $\T^d$, the sphere $\S{d}$,
  and the rotation group $\SO(3)$. While several global approaches are available
  for computing a harmonic expansion from scattered data, their stability and,
  especially, their runtime depend strongly on the geometry of the nodes, in particular
  on a sufficiently small fill distance. Even a single large hole in the data
  may cause instability and long runtimes.
  
  We propose \emph{harmonic approximation via moving least squares} (HAMLS),
  a global harmonic approximation scheme that avoids the ill-conditioned global
  solve and works well in both the overdetermined and the underdetermined
  settings. The main idea is an intermediate transition from the scattered data
  to a quadrature grid, which is realized via moving least squares (MLS). This
  replaces one large global problem by many small local ones that
  can be regularized individually. The global harmonic approximation is then
  obtained via quadrature using a fast Fourier transform. This makes the global
  stage fast, stable and non-iterative. For sufficiently smooth
  functions sampled at well-distributed nodes with small fill distance $h$,
  we bound the resulting $\L2$-error by the error of the harmonic approximation
  obtained from exact values on the quadrature grid, plus a term of order
  $h^{K+1}$, where $K$ is the polynomial degree employed in the MLS step.
  Numerical experiments on $\S{2}$ with ill-distributed nodes show that HAMLS
  achieves errors comparable to those of global least-squares approximation,
  while being significantly faster, especially for larger bandwidths.
\end{abstract} 

\begin{center}
  \begin{minipage}{0.85\textwidth}

  \end{minipage}
\end{center}

\section{Introduction}
\label{sec:intro}

Scattered data approximation arises in many fields, such as geophysical data
analysis, spherical signal processing, computer graphics, and materials science.
Consequently, there are a vast number of global approximation schemes, such
as kernel interpolation, wavelet approximation, or harmonic approximation. In
this work, we consider the torus, the sphere and the rotation group,    
\begin{equation}
  \label{eq:torus_sphere_so3}
  \sX \in \set{\T^{d}, \S{d}, \SO(3)}{d \in \N} , 
\end{equation}
where harmonic approximation provides a particularly natural and versatile
framework for global approximation, enabling fast Fourier algorithms, frequency
analysis, and spectral methods for PDEs.

All three global approximation schemes mentioned above rely on assumptions on the
geometric properties of the given nodes
$X = \set{x_{i}}{}_{i=1}^{N} \subset \sX$ in order to guarantee stability, in
particular on the separation distance $q_{X}$ and on the fill distance $h_{X}$,    
\begin{equation*}
  q_{X} \coloneqq \frac 12 \min\limits_{i \neq j}
  \mathrm{d}_{\sX}(x_{i},x_{j}) , \qquad
  h_{X} \coloneqq \sup\limits_{x \in \sX} \min\limits_{x_{i} \in X}
  \mathrm{d}_{\sX}(x,x_{i}) .
\end{equation*}
For kernel interpolation, Ward and coauthors showed that a sufficiently large 
separation distance implies stability, both on $\IR^{d}$
\cite{Narcowich1991, Ball1992} and on the sphere \cite{Narcowich1998}. Similarly, a
sufficiently small fill distance yields stable global approximation via 
wavelets, as shown by Freeden in \cite{Freeden1998, Freeden1998a}. Harmonic 
approximation usually assumes both a small fill distance and a large separation
distance \cite{Graef2008, Keiner2007}, although the latter condition
can be dropped by using adapted weights, such as Voronoi weights, see Gröchenig
\cite{Groechenig1992, Groechenig2020}, Filbir and Themistoclakis
\cite{Filbir2008}, or Mhaskar, Narcowich and Ward \cite{Mhaskar2001}.

In applications, however, the nodes are often
ill-distributed, with a very dense node population in some regions and large
gaps in others. Thus both conditions on $q_{X}$ and $h_{X}$ are violated, and
obtaining the harmonic approximation can be unstable and time-consuming,
especially at large bandwidths.   

In this work, we propose an algorithm for harmonic approximation of prescribed
bandwidth that also works well for ill-distributed nodes and
avoids the stability-dependent runtime of a global iterative solver.
This is achieved via an intermediate
local step where we first compute approximate values on a quadrature grid, which
then give rise to stable and fast global approximation via a Fourier transform.      

In the following, we explain the mentioned aspects and the algorithm in more detail. For a prescribed bandwidth $L \in \N$, we denote by $\Pi_{L}(\sX)$
the space of functions of bandwidth at most $L$ on $\sX$. On $\T^d$, this is
the span of the Fourier modes with frequency indices in $\{-L,\dots,L\}^d$.
On $\S{d}$ and $\SO(3)$, it is the span of the spherical harmonics and Wigner
$D$-functions, respectively, of degree at most $L$. In each case, these are
eigenfunctions of the Laplace-Beltrami operator on $\sX$. Given an orthonormal
basis $\set{\phi_k}{k\in\mathcal I_L}$ of $\Pi_{L}(\sX)$, indexed by a finite
set $\mathcal I_L$, any band-limited $f \in \Pi_{L}(\sX)$ has the form     
\begin{equation*}
  f = \sum_{k\in\mathcal I_L}\hat f_k\phi_k . 
\end{equation*}
For given nodes $X = \set{x_i}{}_{i=1}^N \subset \sX$ and values
$\vc y = (y_{i})_{i=1}^{N} \in \IC^N$, we want to compute harmonic
coefficients $\hat{\vc f}=(\hat f_k)_{k\in\mathcal I_L}$ such that
$f(x_{i}) \approx y_{i}$ for $i = 1,\dots,N$. With the sampling matrix
$\mt F=(\phi_k(x_i))_{i,k}$ and positive diagonal matrices $\mt W$ and $\mt R$,
the corresponding weighted and regularized least-squares problem reads      
\begin{equation}
  \label{eq:LSQR-Problem_Regularized}
  \min_{\hat{\vc f}\in\IC^{\abs{\mathcal I_L}}}
  \norm{\mt W^{1/2}(\mt F\hat{\vc f}-\vc y)}_2^2
  + \alpha \norm{\mt{R}^{\nicefrac 12} \hat{\vc f}}_{2}^{2},
  \qquad \alpha\geq0.
\end{equation}
Here $\mt W$ weights the contribution of each node. Commonly, Voronoi weights
are employed to damp the influence of densely sampled regions. The parameter
$\alpha$ controls the regularization strength, and $\mt R$ allows higher
degrees to be penalized more strongly. In our experiments, its diagonal
entries are $R_{kk}=(1+\lambda_k)^s$ with Sobolev parameter $s=2$, where
$\lambda_k$ are the nonnegative eigenvalues associated with the eigenfunctions
$\phi_{k}$. Regularization, \ie $\alpha>0$, ensures uniqueness if the system is
underdetermined and can help to stabilize the problem if it is ill-conditioned.

For larger bandwidths $L$, this problem can become quite large. We therefore
solve it iteratively using the LSQR algorithm described in
\cite{Paige1982,Bjoerck2024}, which acts directly on the augmented least-squares
formulation and avoids explicitly forming the normal equations associated with
\eqref{eq:LSQR-Problem_Regularized}, which read 
\begin{equation}
  \label{eq:LSQR_normal_equations}
  \braces*{\mt F^H\mt W\mt F+\alpha\mt R} \hat{\vc f}
  = \mt{F}^{H} \mt W \vc{y} .   
\end{equation}
The conditioning of $\mt F^H\mt W\mt F$ is closely linked to the geometry
of the nodes $X$.
If the nodes and weights provide a quadrature rule that integrates every 
product $\phi_k\overline{\phi_{k'}}$, $k,k'\in\mathcal I_L$, exactly, then
orthonormality of the basis implies $\mt F^{H} \mt W \mt F = \mt I$. Thus no
regularization is needed, \ie $\alpha = 0$, and solving
\eqref{eq:LSQR-Problem_Regularized} simplifies to evaluating
$\hat{\vc f} = \mt F^{H} \mt W \vc y$, which does not require an iterative
solver anymore. Moreover, on suitable quadrature grids we can even employ fast
Fourier-based algorithms for computing $\mt F^{H} \mt W \vc y$.    

For general node sets, Marcinkiewicz-Zygmund (MZ) inequalities compare the
continuous and weighted discrete norms, which allows us to control the
conditioning of $\mt F^H\mt W\mt F$~\cite{Filbir2011}. On $\S{2}$, Voronoi
weights give such bounds when the fill distance $h_X$ is sufficiently small
compared to $1/L$~\cite[Section~3.3]{Keiner2007}. Conversely, uniform MZ
constants over increasing bandwidth $L$ require
$h_X=\sO(L^{-1})$~\cite{Filbir2011}. The geometry of $X$ and the condition of 
$\mt F^{H} \mt W \mt F$ are therefore closely related, and even one large gap
can cause instability. While regularization allows us to stabilize such
an ill-conditioned or underdetermined problem, it introduces bias. Moreover, the
single parameter $\alpha$ must balance stability against oversmoothing over the
whole domain $\sX$, even when the node density varies strongly.        

This motivates \emph{harmonic approximation via moving least squares} (HAMLS), a
global harmonic approximation scheme that circumvents those problems. It first
computes approximate values on a quadrature grid via moving least squares
(MLS)~\cite{Wendland2001b,Wendland2001a,Hielscher2024} 
and then obtains the harmonic expansion from the resulting grid values by an
adjoint fast Fourier transform~\cite{Keiner2009,Mildenberger2023,Potts2009}.
Thus the large global least-squares problem is replaced by many small local
problems and a stable, non-iterative quadrature. Poor local stencils may still
cause ill-conditioned local systems, but their small size allows us to solve
them directly and regularize them individually. To this end, we use a spatially
adapted support radius $\delta(x)$ and a local instability measure tailored to
MLS evaluation, from which we derive a degree-adapted regularization, see
\Cref{sec:mls transfer} and \Cref{sec:mls regularization}. All ingredients are 
chosen to guarantee a certain order of smoothness of the global MLS
approximation. These adaptations of MLS are crucial to the success of HAMLS on
ill-distributed data.     

Our contribution is twofold. First, we formulate HAMLS uniformly on $\T^d$,
$\S{d}$, and $\SO(3)$ and show that, for good node sets and sufficiently smooth
functions, its error with unregularized MLS is bounded by the ideal
quadrature-based approximation error, plus a term of order $h_X^{K+1}$, where
$K$ is the local polynomial degree, see \Cref{thm:hamls_error_bound_good_data}.
This term becomes negligible as the fill distance $h_{X}$ decreases, recovering
the accuracy of the harmonic approximation obtained from exact values on a
quadrature grid. Second, numerical experiments on $\S{2}$ implemented in
MTEX~\cite{Hielscher2007} demonstrate our new algorithm. For ill-distributed
nodes on $\S{2}$, HAMLS attains errors comparable to or smaller than those of
LSQR at a fraction of the runtime, with the advantage increasing with bandwidth. 

In \Cref{sec:hamls}, we present the HAMLS algorithm and explain its steps in
detail, with an emphasis on adapting the MLS step to ill-distributed data. We
close the section with the mentioned error analysis, see
\Cref{thm:hamls_error_bound_good_data}. \Cref{sec:numerics} shows our experiments
on $\S{2}$, with both measured and synthetic data.

\section{The HAMLS Algorithm}  
\label{sec:hamls}

The HAMLS algorithm takes scattered data and a bandwidth $L$ as input. It uses
moving least squares (MLS) to obtain approximate values on a quadrature grid and
then computes the harmonic coefficients from these grid values via a Fourier
transform.
Throughout, $\mt F$ denotes the sampling matrix of the orthonormal
harmonic basis $\{\phi_k\}_{k\in\mathcal I_L}$ introduced above, and $\mt W$
denotes the corresponding diagonal weight matrix. Depending on the setting,
these matrices refer to scattered data $X$ with chosen weights for the
least-squares problem or to quadrature nodes $Q$ with quadrature weights in
the quadrature step of HAMLS, see below.

\begin{algorithm}[htbp]
  \caption{Harmonic Approximation via Moving Least Squares (HAMLS)}
  \label{alg:HAMLS}
  \SetAlgoLined
  \KwInput{\begin{tabular}[t]{p{4cm}l}
    $(x_{i},y_{i})_{i=1}^{N} \in (\sX \times \IC)^{N}$
    & Scattered samples $y_{i}\approx f(x_{i})$ of $f\colon\sX\to\IC$ \\ $L\in\IN$
    & Harmonic bandwidth
  \end{tabular}}
  \begin{itemize}[rightmargin=\algomargin]
    \item[1.] Choose a quadrature grid $Q=\{q_{j}\}_{j=1}^{M_{Q}}\subset\sX$ with
          weights $\omega_{j}>0$ such that the corresponding quadrature rule is
          exact on $\bandlimFunc{2L}{\sX}$. 
    \item[2.] Construct the MLS approximation
          $\sM_{\vc y} \colon \sX\to\IC$ of the data
          $\vc y = (y_{i})_{i=1}^{N}$ and evaluate it at the quadrature
          nodes,
          \begin{equation*}
            f_{j}=\sM_{\vc y}(q_{j}), \qquad j=1,\dots,M_{Q} .
          \end{equation*}
    \item[3.] Compute the harmonic coefficients
          $\hat{\vc f} = \mt F^{H} \mt W \vc{f}$ by the adjoint Fourier transform 
          associated with the quadrature rule.
  \end{itemize}
  \KwOutput{Harmonic approximation $f_{\mathrm{HAMLS}}\in\bandlimFunc{L}{\sX}$
    with $f_{\mathrm{HAMLS}}(x_{i})\approx y_{i}$} 
\end{algorithm}   

\noindent Note that the quadrature rule has to be exact
on $\bandlimFunc{2L}{\sX}$ and not only on $\Pi_{L}(\sX)$, because the
integrands $g\overline{\phi_k}$ of $g\in\bandlimFunc{L}{\sX}$ have bandwidth up
to $2L$. 

In general, the MLS reconstruction $\sM_{\vc y}$ is not band-limited. However,
it inherits the minimal order of smoothness from its weight function, its ansatz
space, and any additional parameters, see \Cref{sec:mls transfer}.
With our construction, $\sM_{\vc{y}}$ has a prescribed smoothness order provided
it exists, thus, in particular, $\sM_{\vc{y}} \in \L2(\sX)$ on the compact domain
$\sX$. Note that step~3 returns the coefficients of $\sM_{\vc y}$ only up to the
aliasing contributions of the higher degrees. It is therefore desirable for 
$\sM_{\vc y} \colon \sX \to \IC$ to possess a certain smoothness order, which
then causes faster decay of the coefficients belonging to those higher-degree
components.  

Altogether, HAMLS replaces one global least-squares problem by many independent
small local problems, which can be solved directly and in parallel, and a
subsequent fast, stable adjoint Fourier transform. Numerical instability is confined
to the small local problems, where regularization can be applied individually
and no iterative solver is required.

\subsection{Efficient Quadrature}\label{sec:quadrature grid}

The general idea relies on computing the global Fourier expansion of the MLS approximation $\sM_{\vc y}$.
Therefore, the harmonic coefficients are computed by quadrature.
The choice of the quadrature grid is of significant interest, since it obviously influences the error of the HAMLS algorithm, but it also has a high impact on the computation time, as evaluating $\sM_{\vc y}$ constitutes most of the computational cost of HAMLS, see~\Cref{tab:TimeHAMLSsteps}.
Hence, the quadrature grid should be kept deliberately small.

On $\T^d$, the equidistant trapezoidal rule with $M_Q=(2L+1)^d$ is exact on $\Pi_{2L}(\T^d)$ and minimal~\cite{Plonka2018}.
In higher dimensions, rank-$1$ lattices~\cite{Kaemmerer2015} or multiple rank-$1$ lattices~\cite{Kaemmerer2018} may be more efficient for index sets other than the full cube.
Minimal rules on $\S{d}$ and $\SO(3)$ are not known in general, although nearly optimal grids~\cite{Graef2013,Graef2011a,Graef2012,Marki2024} exist.

However, we are more interested in controlling the error and therefore
investigate quadrature rules that are exact on $\bandlimFunc{2L}{\sX}$, while
keeping the number of nodes small. Hence, we use product rules obtained from
coordinate representations by applying one-dimensional trapezoidal or Gaussian
quadrature in each coordinate:
\begin{itemize}
  \item On $\S{2}$, trapezoidal quadrature in azimuth and Gauss-Legendre quadrature in the cosine of the polar angle yield $M_Q=(2L+2)(L+1)$ nodes and exactness on $\bandlimFunc{2L}{\S{2}}$~\cite{Graef2009,Graef2009a}.
  \item On $\S{d}$, the same principle applies: the trapezoidal rule treats the single periodic angle, and Gauss-Jacobi rules adapted to the weights $\sin^{k}\vartheta$ treat the $d-1$ polar angles~\cite[Theorem~6.2.3]{DaiXu2013}.
  \item On $\SO(3)$, trapezoidal rules in the Euler angles $\alpha,\gamma$ and
        Gauss-Legendre quadrature in $\cos\beta$ yield $M_Q=(L+1)(2L+1)^2$ nodes and exactness on $\bandlimFunc{2L}{\SO(3)}$~\cite{Khalid2015}.
\end{itemize}
Note that using a Clenshaw-Curtis quadrature rule instead of Gauss-Legendre on $\S{2}$ and $\SO(3)$ would also lead to an exact quadrature rule, but with roughly twice as many nodes, see~\cite{Driscoll1994,Potts1998} and~\cite{Potts2009,Hielscher2026}, respectively.
In return, the grid would be equispaced in all coordinates, so that the Fourier
stage of step~3 can use a standard equispaced FFT on the respective domain.
However, as we will see later in~\Cref{tab:TimeHAMLSsteps}, it is not the Fourier transform but the evaluation of the MLS reconstruction $\sM_{\vc y}$ that dominates the computation time.
Hence, we try to minimize the number of quadrature nodes instead.

We conclude this section by collecting the fast algorithms that realize $\mt F$ and $\mt F^{H}$ on these grids.

On the torus, the computation of the adjoint Fourier transform $\mt F^{H}$ reduces to the standard FFT, since the quadrature grid is equispaced, see~\cite{Plonka2018}.
For nonequispaced data, the nonequispaced fast Fourier transform (NFFT)~\cite{Potts2001} can be used instead.

On the $d$-sphere and on $\SO(3)$, adapted fast spherical and rotational Fourier
transforms are available, see~\cite{Potts1998,Kunis2003,Potts2003},
\cite[Section~5]{Maslen1997} and \cite{Potts2009,Kostelec2008,Risbo1996},
respectively. These approaches essentially map spherical and rotational
functions to ordinary Fourier series on a torus via double coverings,
see~\cite{Mildenberger2022,Hielscher2026,Mildenberger2023}, and subsequently
make use of NFFT~\cite{Potts2009}.

\subsection{Moving Least Squares Approximation} 
\label{sec:mls transfer}

Moving least squares (MLS) approximation was first introduced for scattered data
approximation on $\R^{d}$ in~\cite{Lancaster1981} and further analyzed in
\cite{Levin2000,Wendland2001b,Mirzaei2015}. Later, it was extended to
spheres \cite{Wendland2001a,Hielscher2024} and to general manifolds
\cite{Sober2021}. Besides the data $\vc{y} = (y_{i})_{i=1}^{N}$ at the
nodes $X = \set{x_{i}}{}_{i=1}^{N}$, MLS needs a weight function
$w \colon \sX \times \sX \to \lbrack 0,\infty \rparen$ and a finite-dimensional
ansatz space $\sG$. For a given center $x \in \sX$, assume that
\begin{equation}
  \label{eq:mls_local_problem}
  g_x^* \coloneqq \argmin_{g\in\sG} \sum_{x_i\in X}
  w(x,x_i)\abs{y_i-g(x_i)}^2
\end{equation}
is uniquely defined. The MLS approximation at $x$ is then defined as 
$\sM_{\vc y}(x) \coloneqq g_{x}^{\ast}(x)$.
The weight function $w$ selects the local nodes
\begin{equation}
  \label{eq:mls_local_node_set}
  X(x) \coloneqq \set{x_{i} \in X}{w(x,x_{i}) > 0}
\end{equation}
and the ansatz space $\sG$ determines the local fit. The global MLS
approximation $\sM_{\vc y} \colon \sX \to \IC$ inherits the minimal order of
smoothness of those two ingredients \cite{Levin2000,Wendland2001b}. 

Both $w$ and $\sG$ must be carefully adapted to the domain and the data layout.
Here we consider a radial weight function. For a continuous, nonincreasing
function $\phi \colon \lbrack 0,\infty \rparen \to \lbrack 0,\infty \rparen$
with $\supp(\phi) = [0,1]$ and support radius $\delta > 0$, we set 
\begin{equation}
  \label{eq:weight_function_radial}
  w(x,y) \coloneqq \phi\braces*{\frac{\mathrm{d}_{\sX}(x,y)}{\delta}} ,
\end{equation}
so that $X(x) = B_{\delta}(x) \cap X$, where $B_{\delta}(x)$ is the open
geodesic ball with center $x$ and radius $\delta$. Common choices for $\phi$
are the Wendland functions~\cite{Wendland1995}, which allow $w$ to satisfy a
prescribed smoothness order. For quasi-uniform nodes, a suitable support radius
$\delta > 0$ allows for the error estimates in
\Cref{thm:mls_error_bound_good_data}. For ill-distributed nodes, we let $\delta$
depend on $x$, as described later on.  

Next, we choose the ansatz space $\sG$. Let $K \in \N$ denote its maximal 
polynomial degree. The easiest case is $\sX = \T^{d}$. Since the torus is
locally isometric to $\R^{d}$, we choose $\sG = \Pi_{K}(\R^{d})$ in local
coordinates centered at the evaluation point. Next, for MLS on spheres, it was
shown in~\cite{Hielscher2024} that 
\begin{equation}
  \label{eq:reduced_ansatz_space}
  \Pi_{K,2}(\S{d})
  \coloneqq
  \bigoplus_{\substack{0 \leq k \leq K\\ k \equiv K \pmod*{2}}}\Harm_{k}(\S{d})
\end{equation}
is the natural ansatz space. Surprisingly, only every second 
harmonic degree is needed to attain local approximation order $K+1$. The reduced
size also favors computational efficiency and stability. Finally, the case
$\sX = \SO(3)$ can be reduced to the case of
$\S{3}$ via the unit quaternion representation
$\SO(3) \cong \S{3} \slash\set{\pm 1}{}$, see~\cite{Graef2011}.
Here we also assume that the support radius of MLS is sufficiently small
  that all neighbors have a quaternion representative on the hemisphere
  centered at the representative of the evaluation point. The space
  $\Pi_{K,2}(\S{3})$ is then understood locally using these representatives.
Overall, we have
\begin{equation}
  \label{eq:ansatz_space}
  \sG_K \coloneqq
  \begin{cases}
    \Pi_K(\R^d), & \text{if } \sX=\T^d \\
    \Pi_{K,2}(\S{d}), & \text{if } \sX=\S{d} \\
    \Pi_{K,2}(\S{3}), & \text{if } \sX=\SO(3)  
  \end{cases} . 
\end{equation}

For all three choices of $\sX$, with $\sG$ and $w$ as above, there exists a
classical MLS error estimate. Before introducing it, we recall the
  separation distance $q_{X}$ and the fill distance $h_{X}$ from the
  introduction:
\begin{equation}
  \label{eq:fill_and_separation_distance}
  q_X = \frac{1}{2}\min_{x_i\neq x_j}\mathrm{d}_{\sX}(x_i,x_j) , \qquad
  h_X = \sup_{x\in\sX}\min_{x_i\in X}\mathrm{d}_{\sX}(x,x_i) .
\end{equation}
Proofs of the following theorem on $\T^{d}$ and $\S{d}$ can be
found in \cite{Wendland2001b,Hielscher2024}. The transfer to $\SO(3)$ follows  
from \cite{Graef2011}.

\begin{theorem}[MLS with good data]
  \label{thm:mls_error_bound_good_data}
  Let $\sX \in \set{\T^{d},\S{d},\SO(3)}{d \in \IN}$ and $K \in \IN$ be fixed,
  and choose the ansatz space $\sG=\sG_K$ as in \eqref{eq:ansatz_space}. There
  exists a constant $\rho=\rho(\sX,K)>0$ with the following
  property. 
  For every fixed $R \geq \rho$ and quasi-uniformity constant $c \geq 1$, there
  exist constants $h_0=h_0(\sX,K,R)>0$ and
  $C=C(\sX,K,\phi,c,R)>0$ such that, for every
  $f \in \sC(\sX)$ and every node set $X \subset \sX$ with 
  \begin{equation*}
    h_{X} \leq h_{0}
    \qquad\text{and}\qquad
    \frac{h_{X}}{q_{X}} \leq c ,
  \end{equation*}
  the MLS approximation with weight function \eqref{eq:weight_function_radial}
  and support radius $\delta \equiv R\,h_X$ satisfies 
  \begin{equation}
    \label{eq:mls_error_bound_local}
    \abs{\sM_{f\vert_{X}}(x) - f(x)}
    \leq C \cdot \inf\limits_{g \in \sG} \norm{f - g}_{B_{\delta}(x),\infty} ,
    \qquad x \in \sX .
  \end{equation}
  If additionally $f \in \sC^{K+1}(\sX)$, then there also exists a constant
  $C(f)$ such that
  \begin{equation}
    \label{eq:mls_error_bound_global}
    \norm{\sM_{f\vert_{X}} - f}_{\sX,\infty}
    \leq C(f) \cdot C \cdot h_{X}^{K+1} . 
  \end{equation}
\end{theorem}

For ill-distributed nodes, the conditions on $h_{X}$ and $q_{X}$ need not hold.
Rescaling $w(x,x_{i})$ by the volume of the Voronoi cell belonging to $x_{i}$
allows us to get rid of the quasi-uniformity condition $\nicefrac{h_{X}}{q_{X}} \leq c$
under similar assumptions, but the condition $h_{X} \leq h_{0}$
remains~\cite{Lipman2009}.

\begin{remark}
  The proofs in~\cite{Lipman2009} are stated for MLS on $\R^{d}$. In the
  spherical setting, the final step in the proof of Theorem~3.2 in
  \cite{Lipman2009} needs a small adjustment. There, the ratio of ball volumes
  equals the ratio of their radii raised to the $d$-th power. For spherical
  caps, this equality is no longer valid, but a fixed ratio of radii still gives
  a uniformly bounded volume ratio.
  The other arguments need no adjustment or
  follow directly from known results such as \cite{Hielscher2024}, thus the
  result carries over to spheres and, via $\S{3}$, to $\SO(3)$.
\end{remark}

Another difficulty that arises for ill-distributed nodes is the following. A
constant radius large enough to include at least $\dim(\sG)$ neighbors in 
sparse regions may include far too many in dense regions. We therefore let
$\delta$ depend on $x$. For a chosen
neighbor count $\dim(\sG) \leq n \leq N$ and small $\varepsilon > 0$, we set 
\begin{equation}
  \label{eq:adaptive_support_radius}
  \delta(x) \coloneqq (1+\varepsilon)\cdot\mathrm{d}_n(x),
\end{equation}
where $\mathrm{d}_n(x)$ is the distance from $x$ to its $n$-th nearest
neighbor. The factor $1+\varepsilon$ places at least $n$ nodes strictly inside
the support of $w(x,\cdot)$, ensuring the necessary condition
$\#X(x) \geq \dim(\sG)$ for uniqueness of $g_{x}^{\ast}$ in
\eqref{eq:mls_local_problem}. If $X$ is drawn \iid with respect to some density
on $\sX$, this condition is even sufficient almost surely at each fixed
evaluation point, see \cite{Accio2023}.

\begin{remark}   
  For a family of node sets, the local neighbor count   
  \begin{equation}
    \label{eq:number_of_local_nodes}
    n(x) \coloneqq \#X(x)
  \end{equation}
  can theoretically grow arbitrarily large for any $\varepsilon > 0$, as the 
  separation distance approaches zero. Quasi-uniformity is one possible remedy. 
\end{remark}

Since $\mathrm{d}_n$ is generally only continuous, the resulting MLS
approximation is currently also only continuous. We thus evaluate
$\mathrm{d}_{n}$ on an approximately uniform auxiliary grid, such as a Fibonacci
grid on $\S{2}$, and replace it by its degree-zero MLS approximation
$\widetilde{\mathrm d}_n$ using a suitable constant support radius
proportional to its fill distance. Choosing the auxiliary grid fine
enough ensures $(1+\varepsilon)\widetilde{\mathrm d}_n(x)>\mathrm d_n(x)$ for
all $x\in\sX$ and thus also $\#X(x)\geq n$ even after smoothing.

\subsection{Regularization of the Local MLS Problems}
\label{sec:mls regularization}

The adaptive radius \eqref{eq:adaptive_support_radius} controls the size of the
local problems, but they may still be ill-conditioned, for example if the
local nodes are clustered far away from the center. The local problems therefore  
need regularization. Since MLS does not need the entire local fit 
$g_{x}^{\ast}$, but only its center value $\sM_{\vc y}(x) = g_{x}^{\ast}(x)$,
we regularize with respect to the instability of precisely this value. In the
following, we derive a scalar measure of that instability, relate it to the
Lebesgue constant and to the condition number of the local system, and then
construct a degree-adapted penalty from it. 

\paragraph{Locally adapted ansatz space.}

In order to do that, we first have to revise our choice of the ansatz space. We
want regularized MLS to reproduce constants. This requires $1 \in \sG$, which is
currently violated by the choice \eqref{eq:ansatz_space} for
$\sX \in \set{\S{d},\SO(3)}{d \in \N}$ whenever the maximal degree $K$ is odd.
For regularized MLS on those domains, we therefore use the ansatz space of
polynomials up to degree $K$ in local tangent coordinates. On $\S{d}$, we set 
\begin{equation*}
  \Pi_K^{\mathrm{tan}}(\S{d};x) = \set{p\circ P_x}{p\in\Pi_K(T_x(\S{d}))} ,
\end{equation*}
where $P_x$ is the orthogonal projection onto the tangent space $T_x(\S{d})$. We
therefore assume that the local neighborhoods are contained in the open
hemisphere centered at $x$, \ie $\delta(x) < \frac {\pi}2$ for all
$x \in \S{d}$. On $\SO(3)$, we apply the same construction via the unit
quaternion representation $\SO(3) \cong \S{3} \slash \set{\pm 1}{}$. We employ
this space for all $K$, whether even or odd.
On $\T^d$, the notation $\Pi_K^{\mathrm{tan}}(\T^d;x)$ refers to
the polynomial space $\Pi_K(\R^d)$ in local coordinates centered at $x$.

This new ansatz space still possesses the same dimension and local approximation
order as the old one. For unregularized MLS, the estimates of
\Cref{thm:mls_error_bound_good_data} also remain true with $\Pi_K^{\mathrm{tan}}(\sX;x)$ in place of
$\sG_K$. For more details and proofs, we refer to
\cite[Section~5]{Hielscher2024}.

\paragraph{Center amplification and the Lebesgue constant.}

All quantities below refer to the local problem at one fixed center $x \in \sX$.
For simplicity, we omit the dependence on $x$ in the notation for
the remaining quantities. We further
assume that \eqref{eq:mls_local_problem} has a unique minimizer, so
$\sM_{\vc{y}}(x)$ is well-defined. Let $I \coloneqq \set{i}{w(x,x_{i}) > 0}$ be
the set of local indices, so that $\abs{I} = n(x)$ in the notation of
\eqref{eq:number_of_local_nodes}. We normalize the local weights to
$\sum_{i \in I} w_{i} = 1$, which leaves the local minimizer in
\eqref{eq:mls_local_problem} unchanged. Let $g_{1},\dots,g_{m}$ be a monomial
basis of $\Pi_K^{\mathrm{tan}}(\sX;x)$,
$m = \dim(\Pi_K^{\mathrm{tan}}(\sX;x))$, that is sorted in ascending order by degree, centered at
$x$, and scaled to unit discrete norm
\begin{equation*}
  \sum_{i \in I} w_{i} \abs{g_{j}(x_{i})}^{2} = 1 ,
  \qquad j = 1,\dots,m .
\end{equation*}
Note that we have $g_1 \equiv 1$ and $g_{j}(x) = 0$ for $j \geq 2$. With all
those conventions, we set  
\begin{equation}
  \label{eq:mls local vectors}
  \mt{B} \coloneqq \bigl(\sqrt{w_{i}}\,g_{j}(x_{i})\bigr)_{i \in I,\, j=1,\dots,m} ,
  \qquad
  \mt{H} \coloneqq \mt{B}^{H} \mt{B} ,
  \qquad
  \vc{y}_{w} \coloneqq \bigl(\sqrt{w_{i}}\,y_{i}\bigr)_{i \in I} ,
\end{equation}
where $\mt{H} \in \IC^{m \times m}$ is the corresponding Gram matrix with
respect to the local discrete inner product. Due to normalization, its diagonal 
entries are one. The local problem \eqref{eq:mls_local_problem} now reads
$\min_{\vc c} \norm{\mt{B} \vc c - \vc y_{w}}_{2}^{2}$ with the normal equations 
\begin{equation}
  \label{eq:local_normal_equations}
  \mt{H} \vc c = \mt{B}^{H} \vc y_{w} . 
\end{equation}
Since the basis is centered around $x$, the MLS approximation at $x$ equals the
first coefficient, $\sM_{\vc y}(x) = \vc e_{1}^{\top} \vc c = c_{1}$ with  
$\vc e_{1} = (1,0,\dots,0)^{\top} \in \R^{m}$. We measure the instability of
$\sM_{\vc{y}}(x)$ by the largest change of this value under a unit perturbation
$\vc{\Delta y}$ of the local data. By linearity, we have  
\begin{equation}
  \label{eq:instability_of_center_evaluation}
  \sup_{\norm{\vc{\Delta y}}_{2} \leq 1}
  \abs{\sM_{\vc y + \vc{\Delta y}}(x) - \sM_{\vc y}(x)}
  = \sup_{\norm{\vc{\Delta y}}_{2} \leq 1} \abs{\sM_{\vc{\Delta y}}(x)} .
\end{equation}
Solving \eqref{eq:local_normal_equations} gives
$\sM_{\vc{\Delta y}}(x)
= \vc e_{1}^{\top} \mt{H}^{-1} \mt{B}^{H} \vc{\Delta y}_{w}$, and since
the normalization of the local weights implies 
$\norm{\vc{\Delta y}_{w}}_{2} \leq \norm{\vc{\Delta y}}_{2}$, the
Cauchy-Schwarz inequality yields 
\begin{equation}
  \label{eq:instability_of_center_evaluation_estimate}
  \sup_{\norm{\vc{\Delta y}}_{2} \leq 1} \abs{\sM_{\vc{\Delta y}}(x)}
  \leq \norm{\mt{B} \mt{H}^{-1} \vc e_{1}}_{2}
  = \braces[\big]{\vc e_{1}^{\top} \mt{H}^{-1} \vc e_{1}}^{\nicefrac 12}
  = \sqrt{\bigl[\mt{H}^{-1}\bigr]_{1,1}} .
\end{equation}
We call $\chi = \chi(x) \coloneqq \bigl[\mt{H}^{-1}\bigr]_{1,1}$ the
\emph{center amplification} of the local problem. This is our measure of the
local instability. The bound
\eqref{eq:instability_of_center_evaluation_estimate} is sharp if the
perturbations are instead measured in the weighted norm
$\norm{\vc{\Delta y}}_{w} \coloneqq \norm{\vc{\Delta y}_{w}}_{2}$. In
particular, one can show that $\sqrt{\chi}$ is exactly the norm of the point
evaluation functional $\delta_x\colon\Pi_K^{\mathrm{tan}}(\sX;x)\to\IC$,
$g\mapsto g(x)$ with respect to the local discrete norm on
$\Pi_K^{\mathrm{tan}}(\sX;x)$.

\begin{remark}
  \label{rem:chi vs cond}
  The diagonal entries of $\mt{H}$ equal one,
  thus $\lambda_{\max}(\mt{H}) \geq 1$, and therefore  
  \begin{equation*}
    \chi
    \leq \frac{1}{\lambda_{\min}(\mt{H})}
    \leq \frac{\lambda_{\max}(\mt{H})}{\lambda_{\min}(\mt{H})}
    = \cond(\mt{H}) .
  \end{equation*}
  The converse fails: $\cond(\mt{H})$ can be arbitrarily large while $\chi$
  stays bounded. A condition-based instability measure can therefore introduce
  unnecessary regularization even when $\sM_{\vc{y}}(x)$ is stable.   
\end{remark}

\begin{remark}
  \label{rem:chi vs lebesgue}
  The MLS approximation at $x$ can be written as
  \begin{equation*}
    \sM_{\vc y}(x) = \sum_{i \in I} a_{i}^{\ast}\, y_{i} ,
    \qquad
    a_{i}^{\ast} = \sqrt{w_{i}}\, \overline{\psi_{i}} ,
    \qquad
    \vc\psi \coloneqq \mt{B} \mt{H}^{-1} \vc e_{1} ,
  \end{equation*}
  where $a_{i}^{\ast}$ are the values of the \emph{generating functions} of MLS
  at $x$~\cite{Wendland2001b}. The sum of their absolute values 
  $\Lambda(x) \coloneqq \sum_{i \in I} \abs{a_{i}^{\ast}}$ is the
  \emph{Lebesgue constant} of MLS at $x$. Since $\sum_{i \in I} w_i = 1$, the
  Cauchy-Schwarz inequality yields  
  \begin{equation*}
    \Lambda(x)
    = \sum_{i \in I} \sqrt{w_{i}}\, \abs{\psi_{i}}
    \leq \braces[\bigg]{\sum_{i \in I} w_{i}}^{\nicefrac 12}
    \norm{\vc\psi}_{2}
    = \sqrt{\chi(x)} .
  \end{equation*}
  Therefore, controlling $\chi$ implicitly also controls the Lebesgue constant.
  Due to the standard error bound for unregularized MLS \cite[Theorem~2.5]{Hielscher2024} 
  \begin{equation}
    \label{eq:lebesgue error bound}
    \abs{f(x) - \sM_{f\vert_{X}}(x)}
    \leq \bigl(1 + \Lambda(x)\bigr)
    \inf_{g \in \Pi_K^{\mathrm{tan}}(\sX;x)} \norm{f - g}_{B_{\delta}(x),\infty} ,
  \end{equation}
  this even implies error bounds for MLS.   
\end{remark}

\paragraph{A geometric interpretation.}

We now derive a new expression for $\chi$ that provides more insight and will
later be used to compute the local regularization strength. To this end, we
separate $g_1$ from the other basis functions and write  
\begin{equation*}
  \mt{B} = \begin{pmatrix} \vc v_{1} & \mt{V} \end{pmatrix} ,
  \qquad
  \mt{H} = \begin{pmatrix} 1 & \vc b^{H} \\ \vc b & \mt{C} \end{pmatrix} ,
  \qquad
  \vc b \coloneqq \mt{V}^{H} \vc v_{1} , \quad
  \mt{C} \coloneqq \mt{V}^{H} \mt{V} .
\end{equation*}
The Schur complement formula for the $(1,1)$ entry of the inverse yields
\begin{equation}
  \label{eq:chi via coupling}
  \chi = \frac{1}{1 - r} ,
  \qquad
  r \coloneqq \vc b^{H} \mt{C}^{-1} \vc b
  = \norm{\mt{P}_{\mt{V}}\, \vc v_{1}}_{2}^{2} \in [0,1) ,
\end{equation}
where $\mt{P}_{\mt{V}} \coloneqq \mt{V} \mt{C}^{-1} \mt{V}^{H}$ is the
orthogonal projection onto the image of $\mt{V}$. In particular,
$r = \norm{\mt{P}_{\mt{V}}\vc{v}_{1}}_{2}^{2} \leq \norm{\vc{v}_{1}}_{2}^{2} = 1$ 
and thus $\chi \geq 1$. The \emph{coupling} $r$ measures how well the weighted 
samples of the constant $g_1$ can be approximated by those of the remaining
basis functions. Since $\norm{\vc v_{1}}_{2} = 1$, we can also write  
\begin{equation*}
  \sqrt{\chi} = \frac{1}{\sin\theta} ,
\end{equation*}
where $\theta$ is the angle between $g_1$ and the span of the remaining basis
functions with respect to the local discrete inner product.
As $r$ approaches $1$, the local samples of $g_1$ become nearly
indistinguishable from the span of the remaining basis functions, and 
the center amplification $\chi$ becomes arbitrarily large. Conversely, if  
$r$ stays bounded away from $1$, then $\chi$ stays bounded even if $\mt{C}$ and
hence $\mt{H}$ are arbitrarily ill-conditioned, as noted in \Cref{rem:chi vs
  cond}.   

\paragraph{Degree-adapted regularization.}

We want regularized MLS to reproduce constants. Therefore, regularization will
only act on the lower-right block $\mt C$ of $\mt H$. In order to adapt the
regularization strength to the degree, we compute the Cholesky decomposition      
\begin{equation}
  \label{eq:mls local Cholesky factor}
  \mt{C} = \mt{L} \mt{L}^{H} ,
\end{equation} 
where the factor $\mt{L} \in \IC^{(m-1) \times (m-1)}$ is a lower
triangular matrix with positive diagonal entries. The columns of
$\mt{V}\mt{L}^{-H}$ are orthonormal. Since $\mt{L}^{-H}$ is upper triangular,
the first $j$ columns of $\mt{V}\mt{L}^{-H}$ span the same space as those of
$\mt{V}$, so the hierarchy of the degrees is preserved.
Let $\ell(j) \coloneqq \deg (g_{j+1})$, $j=1,\dots,m-1$, denote the degree of the
$j$-th orthonormalized nonconstant basis function and set    
\begin{equation}
  \label{eq:mls degree multipliers}
  \mt{D} \coloneqq \operatorname{diag}\bigl(\mu_{1},\dots,\mu_{m-1}\bigr) ,
  \qquad
  \mu_{j} \coloneqq \bigl(\ell(j)\,(\ell(j)+d-1)\bigr)^{s} , \quad s = 1 .
\end{equation}
The $\mu_{j}$ are motivated by the eigenvalues of the Laplace-Beltrami operator
on $\S{d}$. On $\SO(3)$, we get a similar expression
$\mu_{j} \coloneqq \bigl(\ell(j)\,(\ell(j)+1)\bigr)^{s}$. For 
$\T^{d}$, we may set $\mu_j = \ell(j)^{2}$, where $\ell(j)$ denotes the total
polynomial degree. This choice is motivated by the quadratic dependence of the
eigenvalues of the torus Laplace-Beltrami operator on frequency and assigns
equal weights to all basis functions of the same degree.
With $\widetilde{\vc{c}} = (c_{2},\dots,c_{m})^{\top}$, the regularized version of
\eqref{eq:mls_local_problem} reads  
\begin{equation}
  \label{eq:regularized MLS problem}
  \min_{c_{1} \in \IC,\ \widetilde{\vc{c}} \in \IC^{m-1}}
  \norm{c_{1} \vc v_{1} + \mt{V} \widetilde{\vc{c}} - \vc y_{w}}_{2}^{2}
  + \tau\, \widetilde{\vc{c}}^{H} \mt{L} \mt{D} \mt{L}^{H}\, \widetilde{\vc{c}} ,
  \qquad
  \tau \geq 0 ,
\end{equation}
and the regularized MLS approximation at $x$ equals the first coefficient
of the minimizer, \ie $\sM_{\vc y,\tau}(x) \coloneqq c_{1}$.

\begin{remark}
    We can now see more precisely why regularized MLS should use polynomials
    up to degree $K$ in the local tangent space, instead of $\sG_{K}$ as in 
    \eqref{eq:ansatz_space}. For odd $K$, the original space
    $\sG_K$ on $\S{d}$ and $\SO(3)$ does not contain the constant function.
    Consider $\S{2} = \set{(u,v,z) \in \R^{3}}{u^{2}+v^{2}+z^{2}=1}$ with $K=1$. 
    Then $\sG_1 = \operatorname{span}\{u,v,z\}$ and at the north pole
    $(0,0,1)^{\top}$ we have $u=v=0$, so the coefficient of $z$ determines the
    value of the MLS approximation at this point. However,
    $z = \sqrt{1-u^2-v^2} = 1 + \sO(u^{2}+v^{2})$ agrees with the constant $1$
    only to first order, but also contains quadratic and higher-order terms. If
    we used $z$ in place of $g_1\equiv1$ in the above construction, leaving its
    coefficient unpenalized would leave these nonconstant terms unpenalized as
    well. In particular, constant data need not be reproduced at the center. The
    space $\Pi_K^{\mathrm{tan}}(\sX;x)$ contains $1$ exactly and thus avoids the
    issue.   
  \end{remark}

\paragraph{Choosing the regularization strength.}

We now discuss how to compute the locally adapted regularization strength
$\tau(x)$. Thus let $x \in \sX$ be variable again. First, we choose a target 
amplification $\chi_{T} > 1$. Whenever $\chi = \chi(x) \leq \chi_{T}$, we set
$\tau(x) = 0$ and perform unregularized MLS. Otherwise, we employ locally
adapted regularization with strength $\tau(x)$ such that the center
amplification of the regularized problem equals $\chi_{T}$. 

Let us make this more precise. The regularized MLS problem \eqref{eq:regularized
  MLS problem} replaces $\mt{C}$ in the normal equations by 
$\mt{C}_{\tau} \coloneqq \mt{L}(\mt I + \tau \mt{D})\mt{L}^{H}$, while the first
row and column of $\mt{H}$ remain unchanged. Writing
$\vc d \coloneqq \mt{L}^{-1} \vc b$, the
regularized coupling at $x$ becomes
\begin{equation}
  \label{eq:regularized MLS amplification}
  r_{\tau}(x) \coloneqq \vc b^{H} \mt{C}_{\tau}^{-1} \vc b
  = \sum_{j=1}^{m-1} \frac{\abs{d_{j}}^{2}}{1 + \tau \mu_{j}} ,
  \qquad
  \chi_{\tau}(x) \coloneqq \frac{1}{1 - r_{\tau}(x)}
  = \bigl[\mt{H}_{\tau}^{-1}\bigr]_{1,1} .
\end{equation}
Here $\mt{H}_{\tau}$ denotes $\mt{H}$ with $\mt{C}$ replaced by $\mt{C}_{\tau}$.
For $\chi(x) > \chi_{T}$, we have $r_{0}(x) > 1 - \chi_{T}^{-1}$, and
$r_{\tau}(x)$ decreases continuously and strictly to zero. Thus there exists a
unique $\tau > 0$ with $\chi_{\tau}(x) = \chi_{T}$, which can be computed via
bisection. Let $\tau(x)$ denote this value.   

This choice produces a discontinuous first derivative at the boundary
  between regions with $\chi(x) \leq \chi_{T}$ and regions with
  $\chi(x) > \chi_{T}$. We therefore introduce an additional parameter
  $\chi_{\mathrm{full}} > \chi_{T}$, at which full regularization is applied,
  and construct the actual local regularization strength $\widetilde{\tau}(x)$
  to smoothly transition between the two regions, with 
  \begin{equation*}
    \begin{cases}
      \widetilde{\tau}(x) = 0, & \text{ if } \chi(x) \leq \chi_{T}, \\
      \widetilde{\tau}(x) = \tau(x), & \text{ if
                                       } \chi(x) \geq \chi_{\mathrm{full}} .
    \end{cases} 
  \end{equation*}
  In our implementation, $\widetilde{\tau}(x)$ is $\sC^{3}$ whenever the weight
  function is at least $\sC^{3}$.

\paragraph{Accuracy and cost.}

The error estimate \eqref{eq:lebesgue error bound} relies on MLS reproducing
the local ansatz space at the center, \ie $\sM_{g\vert_X}(x) = g(x)$ for every 
$g\in\Pi_K^{\mathrm{tan}}(\sX;x)$. Whenever $\tau(x)=0$, this remains unchanged,
but $\tau(x) > 0$ only guarantees reproduction of constants, so approximation
order $K+1$ is no longer guaranteed. Moreover, for $\tau > 0$, we replace the
Lebesgue constant in \eqref{eq:lebesgue error bound} by its upper bound
$\sqrt{\chi_{\tau}(x)}$. Thus, overall, we end up with 
\begin{equation}
  \label{eq:regularized MLS error}
  \abs{f(x) - \sM_{f\vert_{X},\tau}(x)}
  \leq \bigl(1 + \sqrt{\chi_{\tau}(x)}\bigr)
  \inf_{c \in \IC} \norm{f - c}_{B_{\delta}(x),\infty} 
\end{equation}
with $\tau=\widetilde{\tau}(x)$ for the smoothed choice. 
As $\tau \to \infty$, the reconstruction tends to the local weighted mean 
$\sum_{i \in I} w_i y_i$, \ie to Shepard's method~\cite{Shepard1968}.

The amplification $\chi$ and the contributions $\abs{d_j}^{2}$ require the Cholesky
factorization \eqref{eq:mls local Cholesky factor} and one triangular solve,
at a cost of $\complexity{m^3}$ and $\complexity{m^2}$, respectively. Each
bisection step costs $\complexity{m}$ operations. For a fixed bisection
tolerance, the regularization therefore has the same asymptotic cost as the
local solve.

\subsection{Error Analysis}
\label{sec:hamls error}

The upcoming theorem states that for good data, the proposed HAMLS algorithm
yields an error close to that of the harmonic approximation obtained by applying
the quadrature rule to exact values on the quadrature grid, and the gap vanishes
quickly as the fill distance decreases. This is not very surprising and the
proof is quite simple, but nevertheless its assertion provides an insightful
reference.    

We introduce a new notation for the theorem. Let $f\in\sC(\sX)$. For bandwidth
$L$ and a quadrature grid $Q=\set{q_j}{}_{j=1}^{M_Q}$ with weights $\omega_j>0$
that is exact on $\Pi_{2L}(\sX)$, we define     
\begin{equation}
  \label{eq:quadrature_based_approximation}
  f_{Q} \coloneqq \sum_{k\in\mathcal I_L}
  \left( \sum_{j=1}^{M_Q} \omega_j f(q_j) \overline{\phi_k(q_j)} \right)\phi_k
  \;\in\; \Pi_{L}(\sX).
\end{equation}
This is the harmonic approximation produced by applying the quadrature rule to
the exact values of $f$ on $Q$. It is therefore the HAMLS result in the ideal
case in which step~2 reproduces every $f(q_j)$ exactly and provides a natural
reference for error estimates. Note that $f_Q = f$ whenever
$f \in \Pi_{L}(\sX)$. 
The constants $\rho$ and $h_0$ below play the same roles as in
\Cref{thm:mls_error_bound_good_data}, but correspond here to the local  
ansatz space $\Pi_K^{\mathrm{tan}}(\sX;x)$ and may have different values.

\begin{theorem}[HAMLS with good data]
  \label{thm:hamls_error_bound_good_data}
  Let $\sX \in \set{\T^{d},\S{d},\SO(3)}{d \in \IN}$, let $K \in \IN$ be fixed,
  and use the local ansatz space $\Pi_K^{\mathrm{tan}}(\sX;x)$ with weights $w$ as in
  \Cref{thm:mls_error_bound_good_data}, with $\delta \equiv R\,h_X$ for fixed
  $R \geq \rho$ and quasi-uniformity constant $c \geq 1$. For a fixed harmonic
  bandwidth $L \in \IN$, let $(q_j,\omega_j)_{j=1}^{M_Q}$ be a quadrature
  rule with positive weights that is exact on $\Pi_{2L}(\sX)$.
  Then, for every $f \in \sC^{K+1}(\sX)$ and every node set $X \subset \sX$ with
  \begin{equation*}
    h_X \leq h_0
    \qquad\text{and}\qquad
    \frac{h_X}{q_X} \leq c ,
  \end{equation*}
  the approximation $f_{\mathrm{HAMLS}}$ computed from the exact samples
  $\set{(x_i,f(x_i))}{}_{x_i\in X}$ using unregularized MLS satisfies
  \begin{equation*}
    \norm{f-f_{\mathrm{HAMLS}}}_{\L2(\sX)}
    \leq \norm{f - f_{Q}}_{\L2(\sX)} + C(\sX, K, \phi, c, R, f) \cdot h_{X}^{K+1} .
  \end{equation*}
  The constant $C(\sX,K,\phi,c,R,f)$ is independent of $L$ and the quadrature rule.
\end{theorem}

\begin{proof}
  First of all, note that \Cref{thm:mls_error_bound_good_data} remains
    true for the locally adapted ansatz space $\Pi_{K}^{\mathrm{tan}}(\sX;x)$,
    with different values for the constants involved \cite{Hielscher2024}.
  Inserting $f_Q$ and applying the triangle inequality yields 
  \begin{equation*}
    \norm{f - f_{\mathrm{HAMLS}}}_{\L2(\sX)}
    \leq \norm{f - f_Q}_{\L2(\sX)} + \norm{f_Q - f_{\mathrm{HAMLS}}}_{\L2(\sX)} .
  \end{equation*}
  It therefore remains to bound the second term. Let $\vc y\coloneqq f\vert_X$ 
  be exact samples and let $\sM_{\vc y}$ be the corresponding MLS approximation. 
  We denote by 
  \begin{equation*}
    f\big\vert_Q \coloneqq (f(q_j))_{j=1}^{M_Q} , \qquad
    \sM_{\vc y}\big\vert_Q \coloneqq (\sM_{\vc y}(q_j))_{j=1}^{M_Q}
  \end{equation*}
  the vectors of values of $f$ and $\sM_{\vc{y}}$ on $Q$, respectively. The 
  harmonic coefficient vectors of $f_Q$ and $f_{\mathrm{HAMLS}}$ are then 
  $\mt F^H\mt W f\big\vert_Q$ and $\mt F^H\mt W\sM_{\vc y}\big\vert_Q$. Since
  both $f_Q$ and $f_{\mathrm{HAMLS}}$ belong to $\Pi_L(\sX)$ and the harmonic
  basis is orthonormal in $\L2(\sX)$, Parseval's identity yields  
  \begin{equation*}
      \norm{f_Q - f_{\mathrm{HAMLS}}}_{\L2(\sX)}
      = \norm{\mt F^H\mt W\bigl(f\big\vert_Q - \sM_{\vc y}\big\vert_Q\bigr)}_2 
      \leq \norm{\mt F^H\mt W^{1/2}}_{2\to2} \cdot 
      \norm{\mt W^{1/2}\bigl(f\big\vert_Q - \sM_{\vc y}\big\vert_Q\bigr)}_2 . 
  \end{equation*}
  Here we also used that the quadrature rule is exact on $\Pi_{2L}(\sX)$, and 
  thus $\mt F^H\mt W\mt F=\mt I$. Together with
  $\sum_{j=1}^{M_Q}\omega_j=\mathrm{vol}(\sX)$, we finally obtain  
  \begin{equation*}
    \begin{aligned}
      \norm{\mt W^{1/2}\bigl(f\big\vert_Q - \sM_{\vc y}\big\vert_Q\bigr)}_2
      &= \left(\sum_{j=1}^{M_Q}\omega_j
      \abs{f(q_j)-\sM_{\vc y}(q_j)}^2\right)^{\nicefrac 12} \\
      &\leq \sqrt{\mathrm{vol}(\sX)}\,\norm{f-\sM_{\vc y}}_{\sX,\infty} \\
      &\leq \sqrt{\mathrm{vol}(\sX)}\,C(f)\,C(\sX,K,\phi,c,R)\,h_X^{K+1} ,
    \end{aligned}
  \end{equation*}
  where the last inequality follows from \Cref{thm:mls_error_bound_good_data}.
  This proves the claim. 
\end{proof}

The first term is the error of the harmonic approximation from exact grid
values and vanishes as $L\to\infty$ for sufficiently smooth $f$. The MLS
error is bounded by a term of order $h_X^{K+1}$, independently of $L$. 
Increasing the bandwidth therefore does not improve this bound for a fixed node
set. For fixed $L$ and $h_X\to0$ under the conditions of
\Cref{thm:mls_error_bound_good_data}, the HAMLS approximation approaches $f_Q$.

\section{Numerical Experiments}
\label{sec:numerics}

We evaluate HAMLS on $\S{2}$ against the LSQR solution of the
regularized problem~\eqref{eq:LSQR-Problem_Regularized}, referred to below as
the LSQR approximation~\cite{Paige1982}.

The implementations rely on FFTW 3.3.10~\cite{Frigo2021}, NFFT
3.5.3~\cite{Keiner2009}, and MTEX 7.1.0~\cite{Hielscher2007}. 
They are parallelized and make use of multi-core architectures.
All experiments were conducted on a $3.8\,\mathrm{GHz}$ AMD Ryzen$^{\text{TM}}$
7 5800X CPU with 8 cores and $128\,\mathrm{GB}$ of RAM, using double-precision 
arithmetic. The code and data are available at
\url{https://github.com/mtex-toolbox/mtex-paper/tree/master/HarmonicApproximationViaMovingLeastSquares}.  

Both methods use the Voronoi areas of the sampling nodes: as
least-squares weights in LSQR and to rescale the local MLS weights in HAMLS.
Unless stated otherwise, reported total runtimes include preprocessing,
in particular the computation of these weights. The time budgets in 
\Cref{fig:Error_x*Time} exclude this common Voronoi cost.

In this paper, we compute all spherical Fourier transforms by using the fast NFSFT algorithm~\cite{Potts1998,Kunis2003,Potts2003} from the NFFT toolbox.
In this implementation, we set the threshold parameter controlling the stability of the fast polynomial transform (FPT) to $\kappa=1000$.
For the associated NFFT, we use an oversampling factor of $2$, a cut-off parameter of $6$, and the Kaiser-Bessel window function.

The HAMLS algorithm on $\SO(3)$ is also implemented in MTEX, and a corresponding
rotational approximation example is described in a paper in preparation by
Dittes et al.

\subsection{Weather Data on \texorpdfstring{$\S{2}$}{S²}}\label{sec:weather data}

Our first, qualitative example uses the $N=6501$ daily mean temperatures
recorded by the Global Historical Climatology Network~\cite{Menne2012} on
Fourier's 250th birthday, see \Cref{fig:S2WeatherData}. The
stations cluster over the continents and leave large ocean regions almost
unsampled, so that the node set is strongly ill-distributed, with fill distance
$h_X \approx 29.7^\circ$. We will see in
\Cref{fig:S2WeatherIterations} that this indeed causes the expected long runtime 
of the iterative LSQR solver, due to the bad conditioning of the system matrix.   

The HAMLS algorithm has a fixed running time once its parameters are chosen,
while both the error and the runtime of the LSQR approximation depend on the
number of iterations and hence on the stopping criterion. In the example shown
in \Cref{fig:S2Weather}, HAMLS takes $8.8$ seconds on our system while the LSQR
approximation is being iterated, with the time budget specified as (d) the same
amount of time, (e) five times that amount, (f) ten times that amount, and (g)
thirty times that amount. The first two of these fits are
visually far from the minimizer of~\eqref{eq:LSQR-Problem_Regularized}.

\newcommand{\weatherPanel}[1]{%
  \includegraphics[width=\linewidth]{#1}}

\begin{figure}[htbp]
  \centering
  \begin{minipage}[c]{0.92\textwidth}
    \centering
    \begin{subfigure}[t]{0.31\textwidth}
      \centering
      \weatherPanel{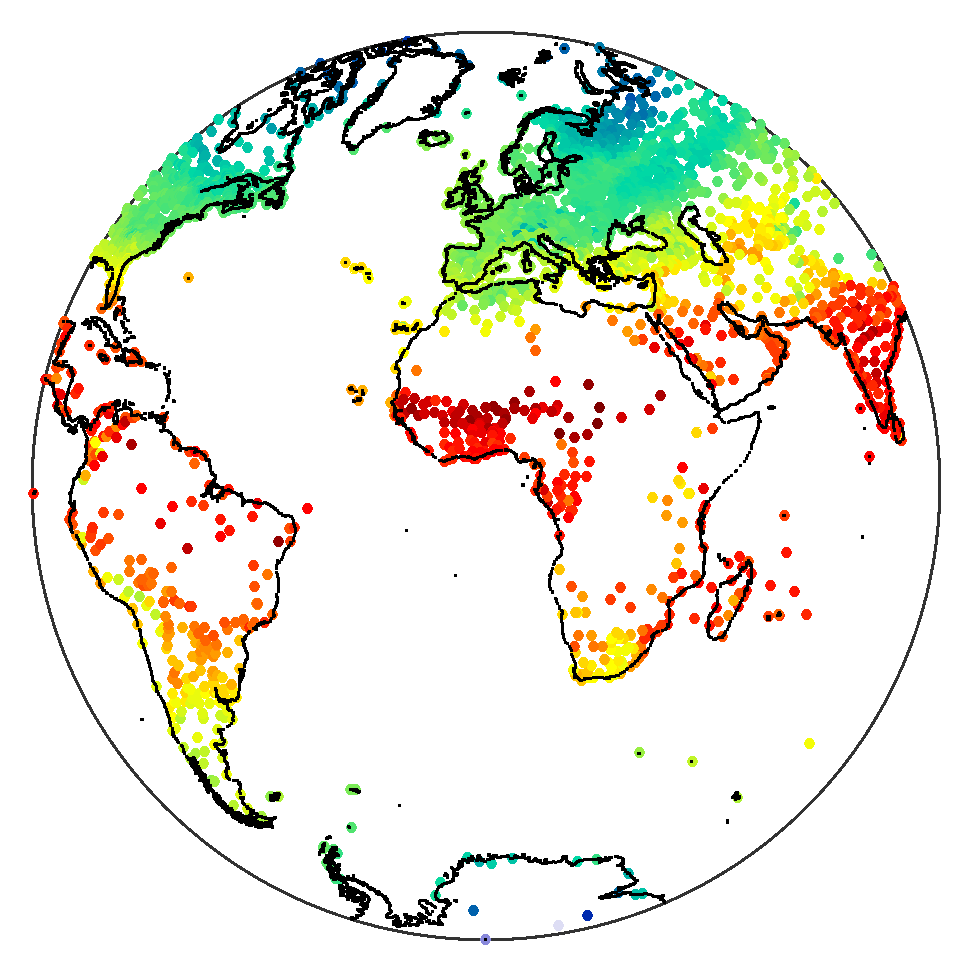}
      \subcaption{Weather data set}\label{fig:S2WeatherData}
    \end{subfigure}
    \begin{subfigure}[t]{0.31\textwidth}
      \centering
      \weatherPanel{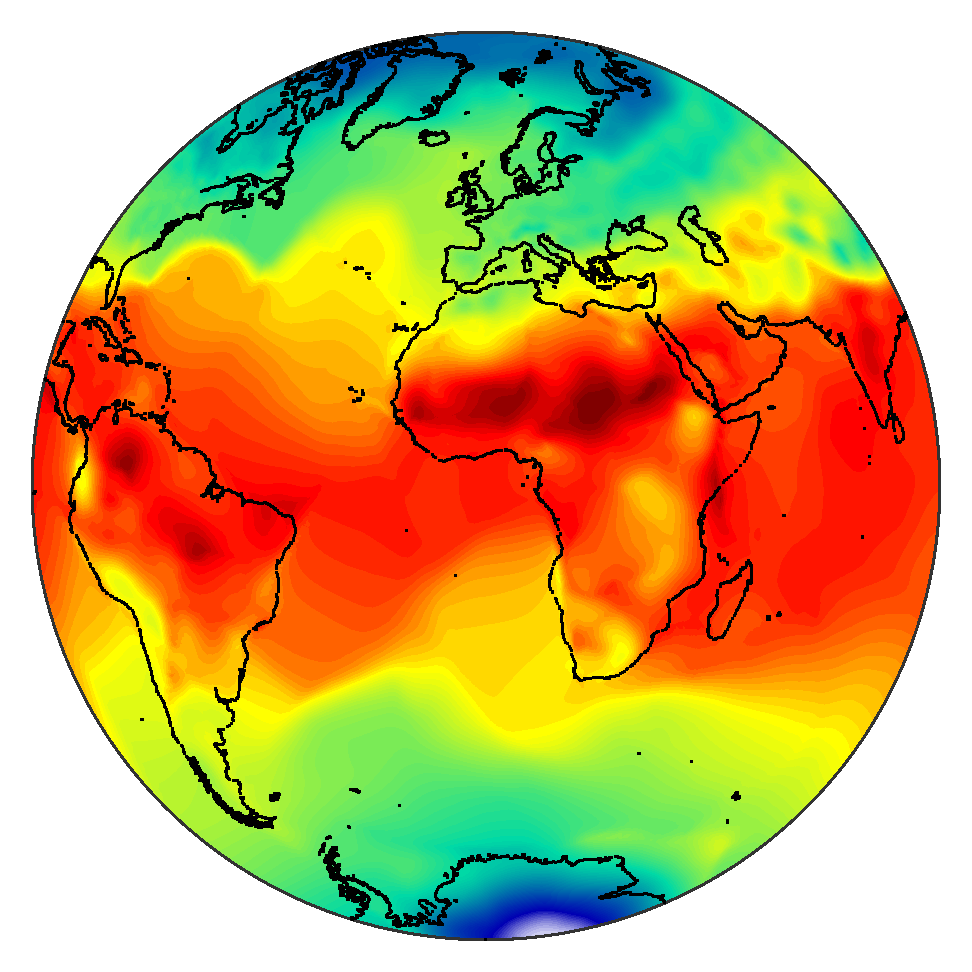}
      \subcaption{MLS reconstruction}\label{fig:S2WeatherMLS}
    \end{subfigure}
    \begin{subfigure}[t]{0.31\textwidth}
      \centering
      \weatherPanel{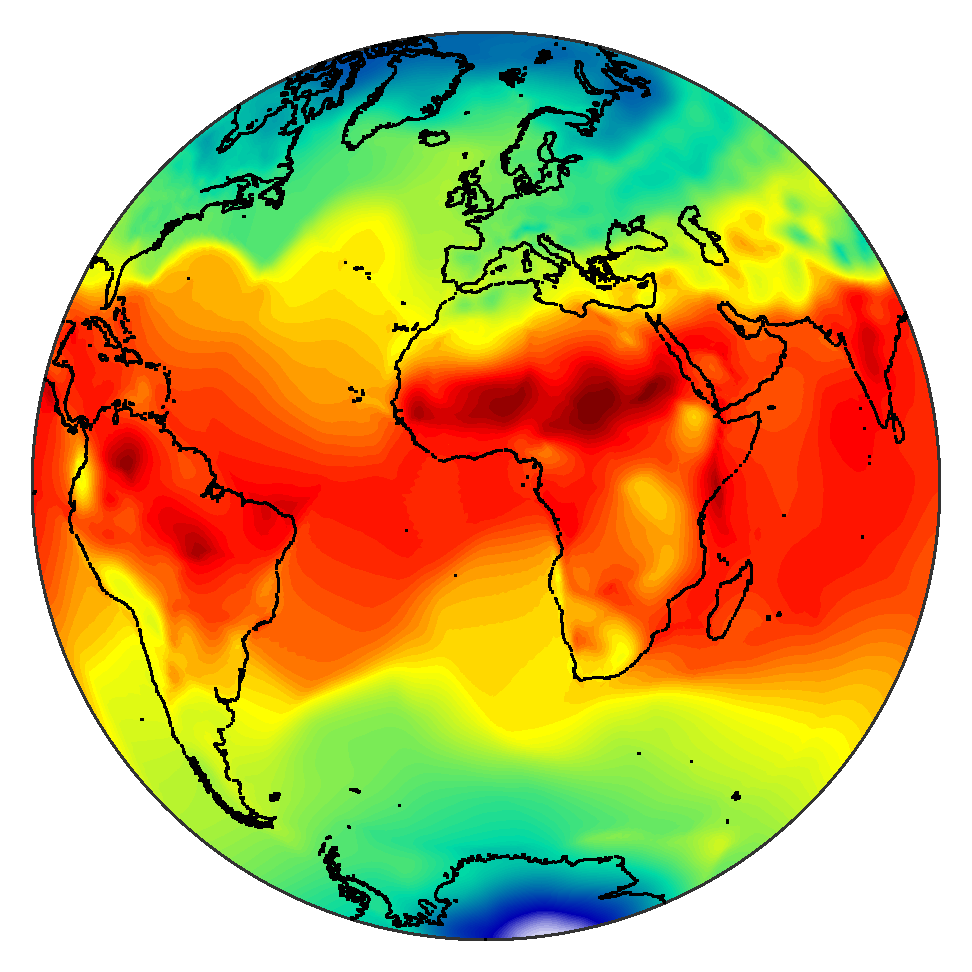}
      \subcaption{HAMLS approximation}\label{fig:S2WeatherHAMLS}
    \end{subfigure}\\[1em]
    \begin{subfigure}[t]{0.2325\textwidth}
      \centering
      \weatherPanel{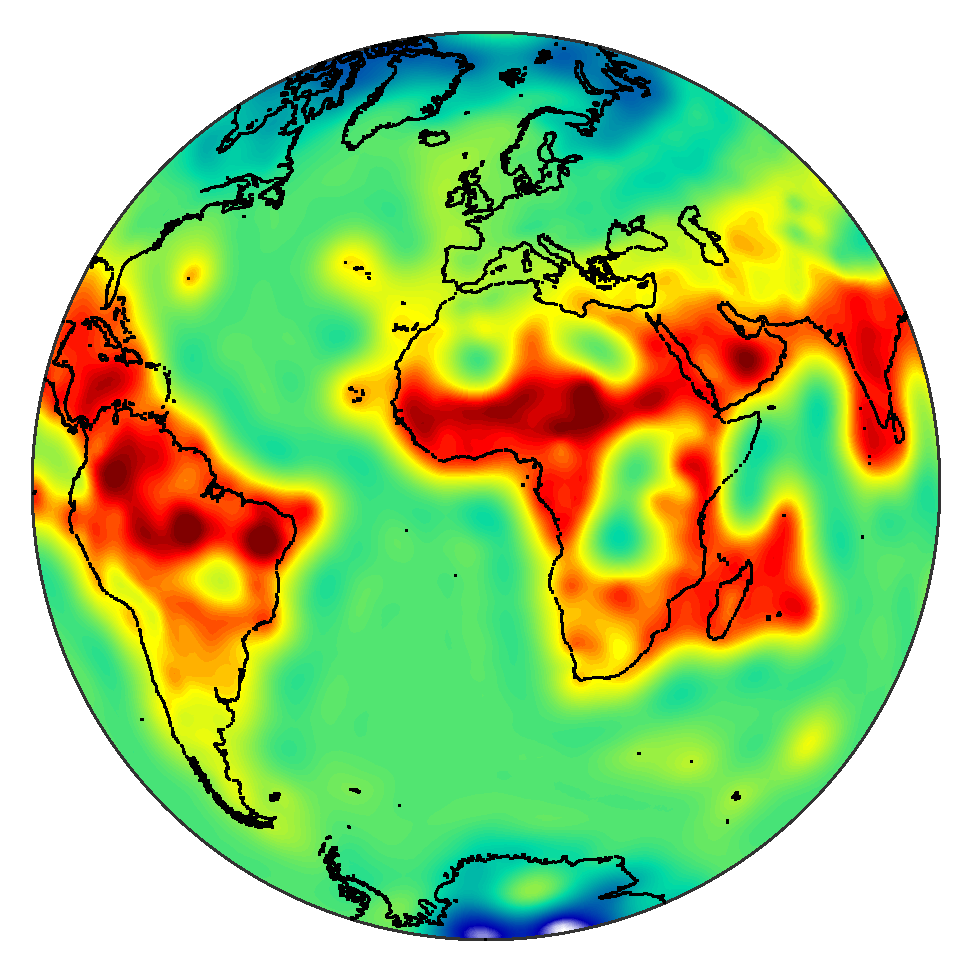}
      \subcaption{LSQR ($280$ iter.)}
    \end{subfigure}
    \begin{subfigure}[t]{0.2325\textwidth}
      \centering
      \weatherPanel{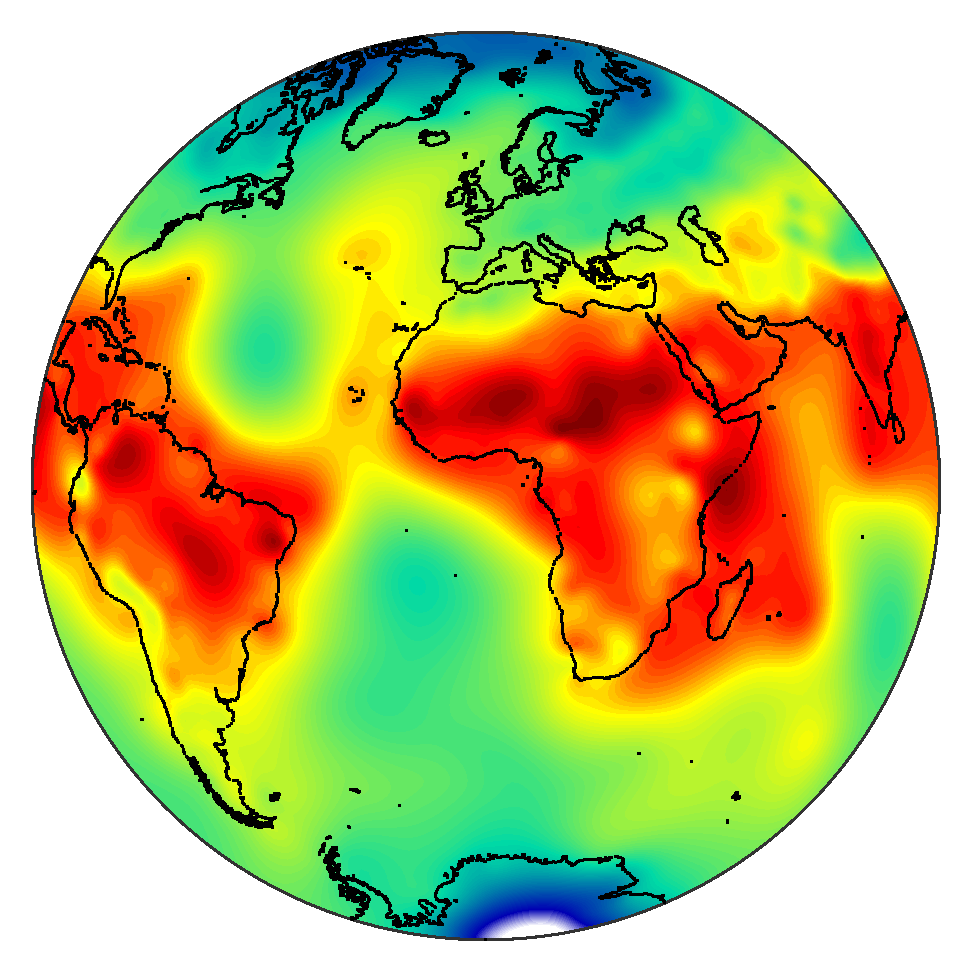}
      \subcaption{LSQR ($1410$ iter.)}
    \end{subfigure}
    \begin{subfigure}[t]{0.2325\textwidth}
      \centering
      \weatherPanel{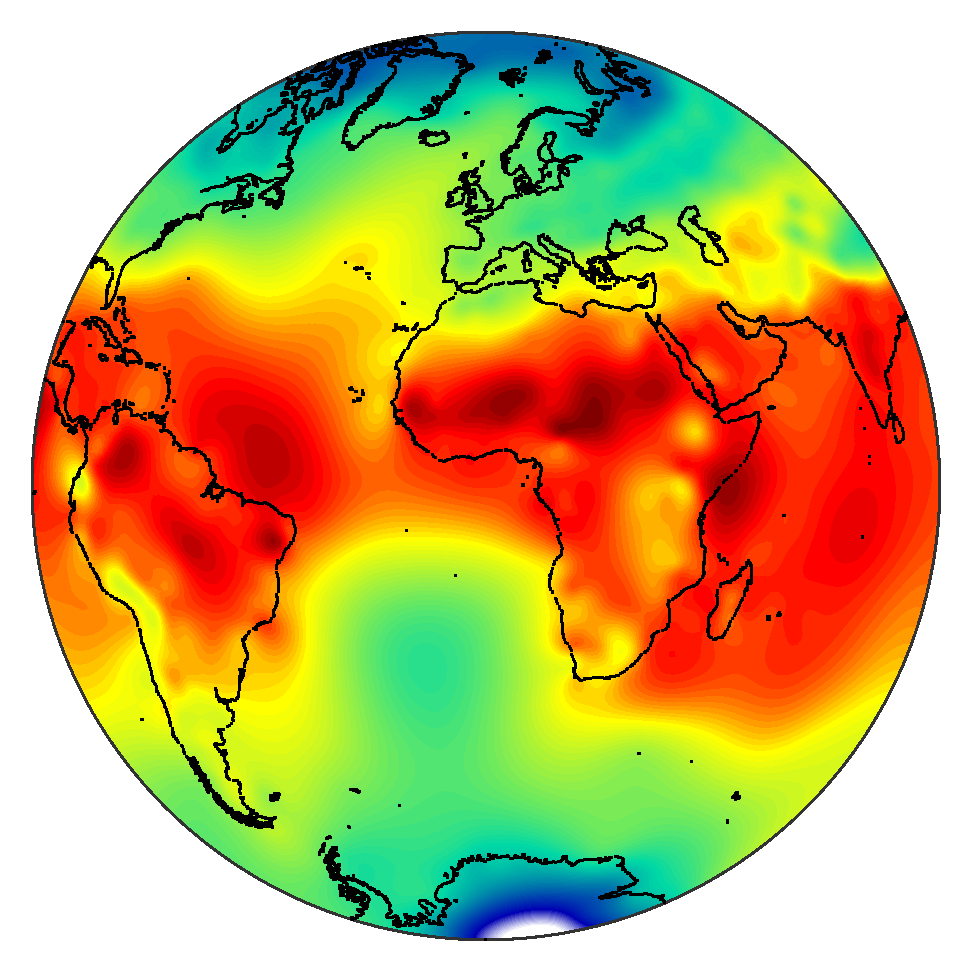}
      \subcaption{LSQR ($2823$ iter.)}
    \end{subfigure}
    \begin{subfigure}[t]{0.2325\textwidth}
      \centering
      \weatherPanel{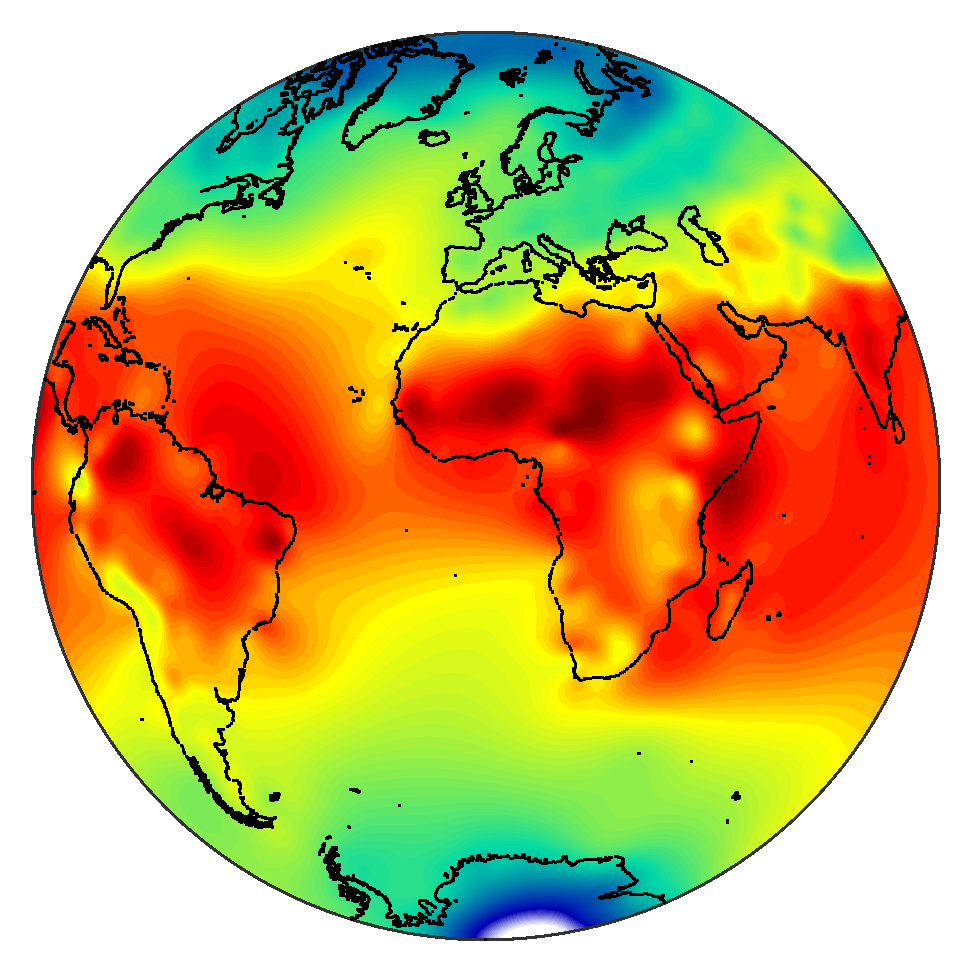}
      \subcaption{LSQR ($8475$ iter.)}\label{fig:S2WeatherLSQR}
    \end{subfigure}
  \end{minipage}
  \hfill
  \begin{minipage}[c]{0.07\textwidth}
    \centering
    \includegraphics[height=9cm, trim={600bp 195bp 12bp 0bp}, clip]{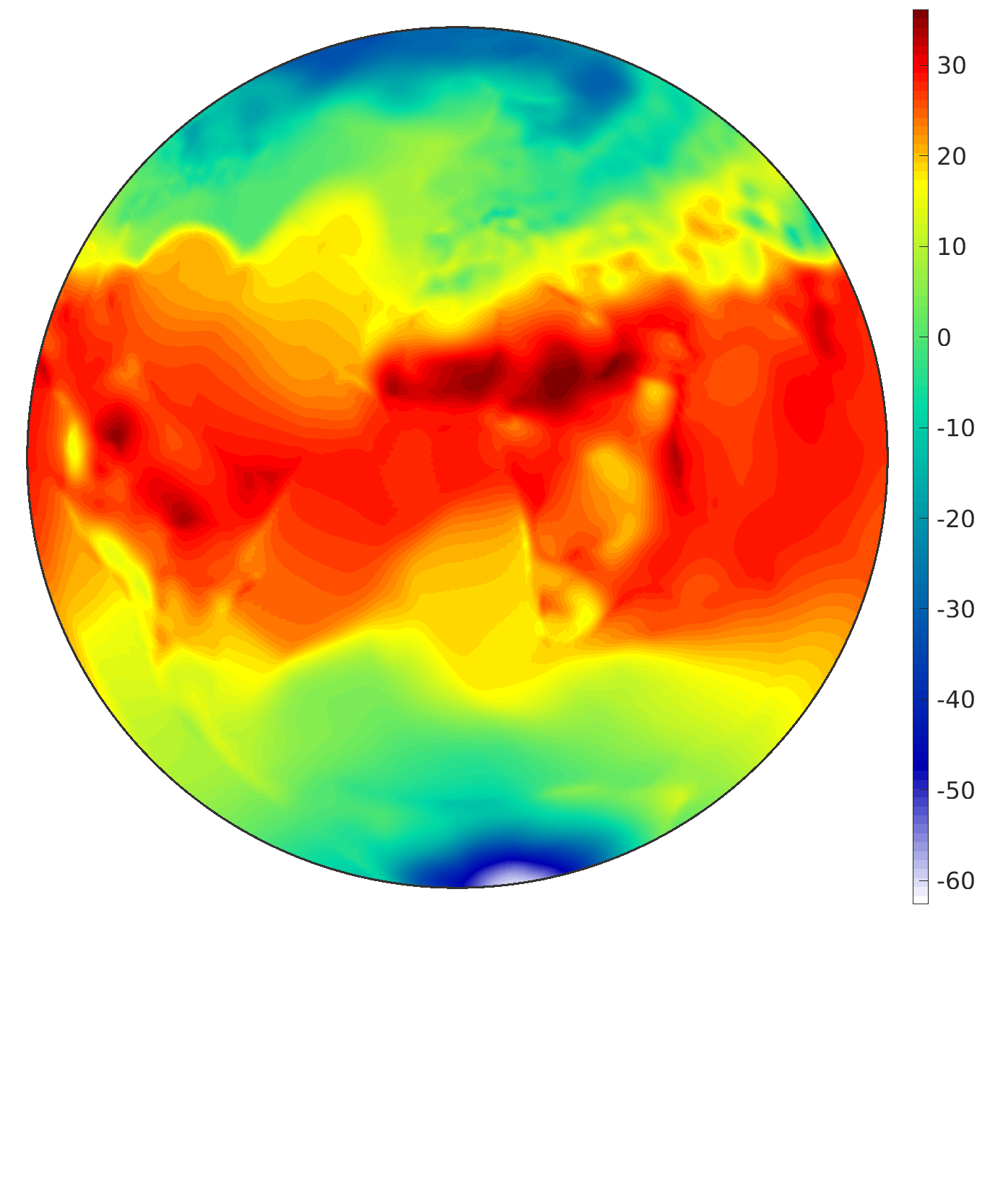}
  \end{minipage}
  \caption{Stereographic projection of the hemisphere centered on the prime
    meridian, showing the weather data~\cite{Menne2012}, the MLS reconstruction of
    degree $4$ with target neighbor count $n = 60$, and
    its HAMLS expansion of bandwidth $L=256$.
    This is compared with the LSQR approximations of bandwidth $L=256$, obtained by
    minimizing the functional~\eqref{eq:LSQR-Problem_Regularized},
    weighted by the Voronoi areas, with regularization parameter
    $\alpha=6.3\cdot10^{-9}$ and Sobolev index $2$. 
    The four LSQR fits take $8.9$, $45.8$, $89.1$ and $267.7$ seconds.
    All panels share the color scale on the right, in
    ${}^\circ\mathrm{C}$.}\label{fig:S2Weather} 
\end{figure}

The LSQR approximation plots show that a large number of LSQR iterations is required to obtain a good solution.
For a closer examination, \Cref{fig:S2WeatherIterations} displays the evolution of the two terms of the minimization functional~\eqref{eq:LSQR-Problem_Regularized} for different regularization parameters.
The penalty term continues to decrease over thousands of iterations, even though the data misfit has settled after a few hundred.
Hence, the extra runtime is spent not on improving the fit at the stations, but on removing the roughness that the early iterates leave over the unsampled oceans.

\begin{figure}[htbp]
  \centering
  \ref{iterlegend}\par\medskip
  \begin{subfigure}[t]{0.49\textwidth}
  \begin{tikzpicture}
    \begin{axis}[
      name = mainaxis,
      width=\textwidth,
      height=0.65\textwidth,
      xmode=log, ymode=log, log basis y={2},
      xmin=9, xmax=9000,
      ymin=0.2, ymax=20,
      ytick={0.25,0.5,1,2,4,8,16},
      yticklabels={0.25,0.5,1,2,4,8,16},
      minor tick style={draw=none},
      ymajorgrids,
      xmajorgrids,
      font={\scriptsize},
      legend to name=iterlegend,
      legend columns=4,
      legend style={draw=black, font=\scriptsize,
        /tikz/every even column/.append style={column sep=0.4cm}},
      xlabel={LSQR iterations},
      ylabel={Data misfit},
      ]

      \addplot[color=blue, solid, mark=*, mark size=1.5pt] table
        [x=iterations, y=misfit, col sep=space]
        {code/weather/results/IterationStudy_reg_1em07.txt};
      \addlegendentry{$\alpha=10^{-7}$}

      \addplot[color=Green, solid, mark=triangle*, mark size=2pt] table
        [x=iterations, y=misfit, col sep=space]
        {code/weather/results/IterationStudy_reg_1em08.txt};
      \addlegendentry{$\alpha=10^{-8}$}

      \addplot[color=orange, solid, mark=square*, mark size=1.5pt] table
        [x=iterations, y=misfit, col sep=space]
        {code/weather/results/IterationStudy_reg_1em09.txt};
      \addlegendentry{$\alpha=10^{-9}$}

      \addplot[color=red, solid, mark=star, mark size=2.5pt] table
        [x=iterations, y=misfit, col sep=space]
        {code/weather/results/IterationStudy_reg_1em10.txt};
      \addlegendentry{$\alpha=10^{-10}$}
    \end{axis}

    \begin{axis}[
      at={(mainaxis.south west)}, anchor=south west,
      width=\textwidth,
      height=0.65\textwidth,
      axis x line*=top,
      axis y line=none,
      xmode=log, ymode=log,
      xmin=9, xmax=9000,
      ymin=0.2, ymax=20,
      xtick={32,318,3180},
      xticklabels={$1$,$10$,$100$},
      minor tick style={draw=none},
      font={\scriptsize},
      xlabel={Runtime in seconds},
      ]
      \addplot[draw=none, forget plot] coordinates {(9,0.2)};
    \end{axis}
  \end{tikzpicture}
  \end{subfigure}
  \begin{subfigure}[t]{0.49\textwidth}
  \begin{tikzpicture}
    \begin{axis}[
      name=mainaxis,
      width=\textwidth,
      height=0.65\textwidth,
      xmode=log, ymode=log, log basis y={2},
      xmin=9, xmax=9000,
      ymin=0.2, ymax=20,
      ytick={0.25,0.5,1,2,4,8,16},
      yticklabels={0.25,0.5,1,2,4,8,16},
      minor tick style={draw=none},
      ymajorgrids,
      xmajorgrids,
      font={\scriptsize},
      xlabel={LSQR iterations},
      ylabel={Penalty term},
      ]

      \addplot[color=blue, solid, mark=*, mark size=1.5pt] table
        [x=iterations, y=penalty, col sep=space]
        {code/weather/results/IterationStudy_reg_1em07.txt};

      \addplot[color=Green, solid, mark=triangle*, mark size=2pt] table
        [x=iterations, y=penalty, col sep=space]
        {code/weather/results/IterationStudy_reg_1em08.txt};

      \addplot[color=orange, solid, mark=square*, mark size=1.5pt] table
        [x=iterations, y=penalty, col sep=space]
        {code/weather/results/IterationStudy_reg_1em09.txt};

      \addplot[color=red, solid, mark=star, mark size=2.5pt] table
        [x=iterations, y=penalty, col sep=space]
        {code/weather/results/IterationStudy_reg_1em10.txt};
    \end{axis}

    \begin{axis}[
      at={(mainaxis.south west)}, anchor=south west,
      width=\textwidth,
      height=0.65\textwidth,
      axis x line*=top,
      axis y line=none,
      xmode=log, ymode=log,
      xmin=9, xmax=9000,
      ymin=0.2, ymax=20,
      xtick={32,318,3180},
      xticklabels={$1$,$10$,$100$},
      minor tick style={draw=none},
      font={\scriptsize},
      xlabel={Runtime in seconds},
      ]
      \addplot[draw=none, forget plot] coordinates {(9,0.2)};
    \end{axis}
  \end{tikzpicture}
  \end{subfigure}
  \caption{The two terms of the regularized problem~\eqref{eq:LSQR-Problem_Regularized} for the weather data at bandwidth $L=256$ as a function of the number of LSQR iterations:
    the data misfit $\norm{\mt W^{1/2}(\mt F\hat{\vc f}-\vc y)}_2$ (left) and
    the penalty term
    $\sqrt\alpha\,\norm{\mt{R}^{\nicefrac 12} \hat{\vc f}}_{2}$ (right) for
    different regularization parameters $\alpha$. 
    The weights $\mt W$ are the normalized Voronoi areas of the stations, so that the data misfit corresponds to a weighted root-mean-square error in ${}^\circ\mathrm{C}$.
    Here, $\mt R$ represents the Sobolev regularization with Sobolev index $2$.}\label{fig:S2WeatherIterations}
\end{figure}
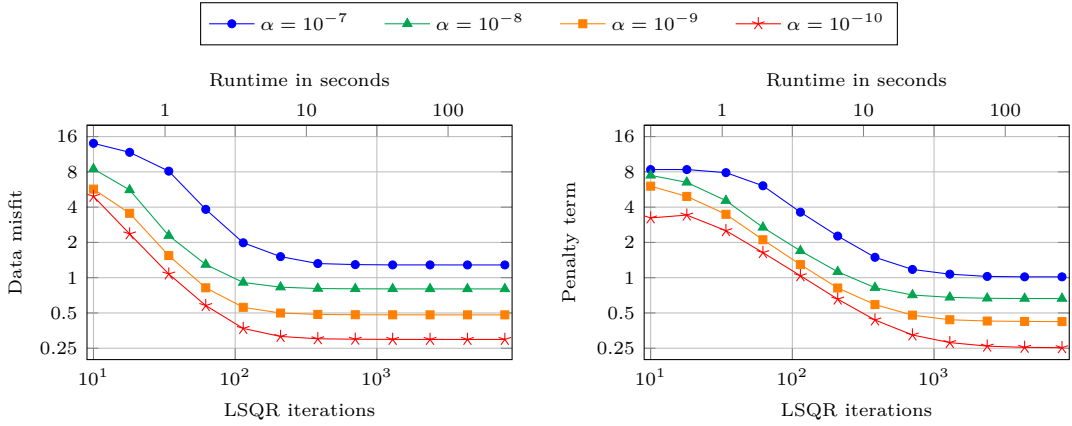

\FloatBarrier
\subsection{A Synthetic Example on \texorpdfstring{$\S{2}$}{S²}}
\label{sec:toy example S2}

Since the previous example did not provide a ground truth against which the approximation error could be assessed, we now consider a synthetic test function $f_{0}$, shown in \Cref{fig:S2ToyFunction}.
The test function has a dominant constant harmonic coefficient, while the remaining coefficients are comparatively small.
This results in a strongly uneven distribution of the harmonic coefficients and, together with the ill-distributed sampling nodes, can make their recovery by LSQR numerically challenging.

We draw $N=10^4$ nodes $\{x_{i}\}_{i=1}^{N}$, with $96\,\%$ sampled from a
probability density constructed from a smoothed version of $f_0-300$ and the
remaining $4\,\%$ sampled uniformly on $\S{2}$, as shown in
\Cref{fig:S2ToyNodes}. The nodes therefore concentrate in regions where $f_0$
exhibits structure and are strongly non-uniformly distributed, with a fill
distance of $h_X\approx14.6^\circ$. The uniformly sampled fraction
prevents the gaps from becoming arbitrarily large.

\newlength{\toyPanelSize}
\setlength{\toyPanelSize}{0.38\textwidth}

\begin{figure}[htbp]
  \centering
  \begin{minipage}[c]{0.80\textwidth}
    \centering
    \begin{subfigure}[t]{0.475\linewidth}
      \centering
      \includegraphics[width=\linewidth]{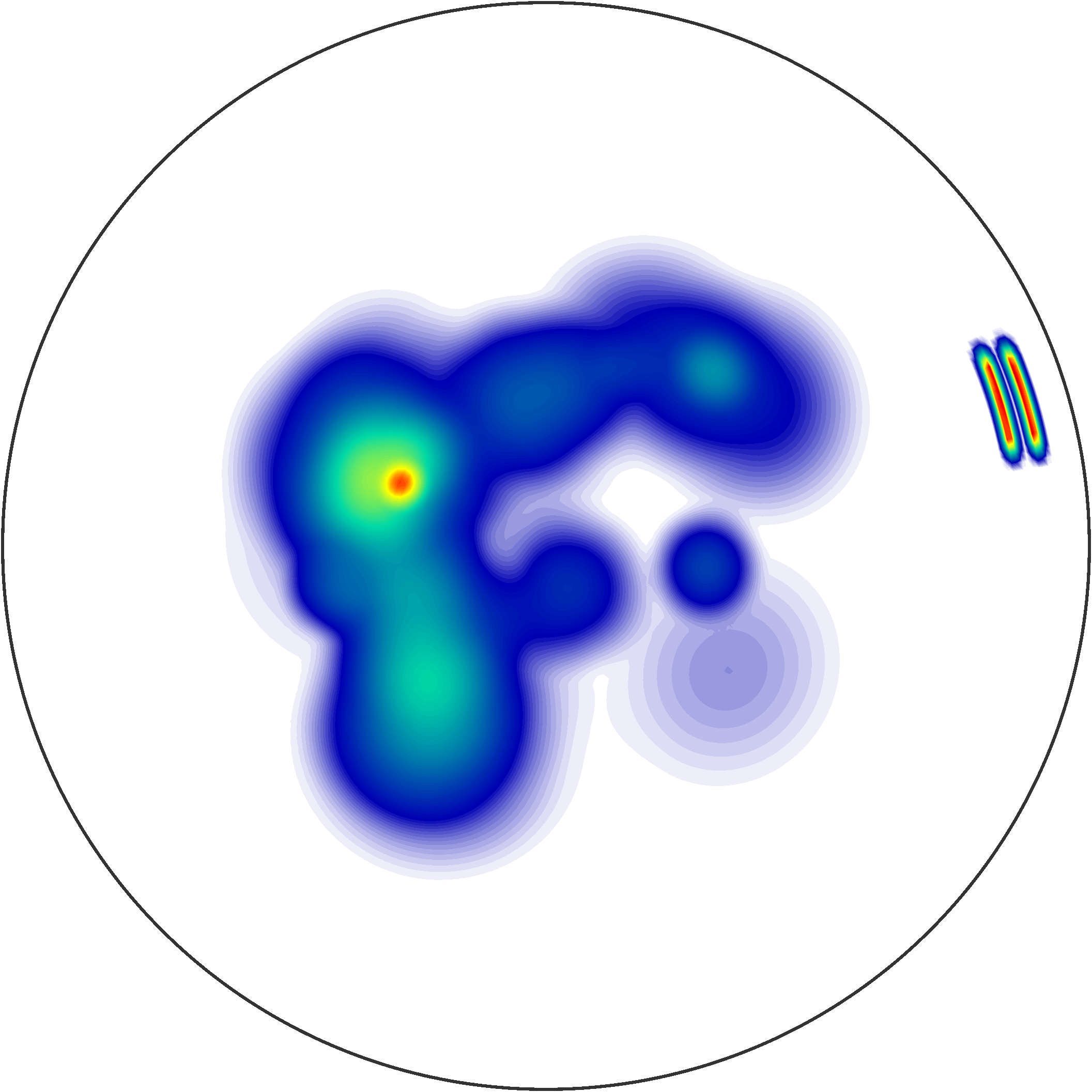}
      \subcaption{Test function $f_{0}$}\label{fig:S2ToyFunction}
    \end{subfigure}
    \hfill
    \begin{subfigure}[t]{0.475\linewidth}
      \centering
      \includegraphics[width=\linewidth]{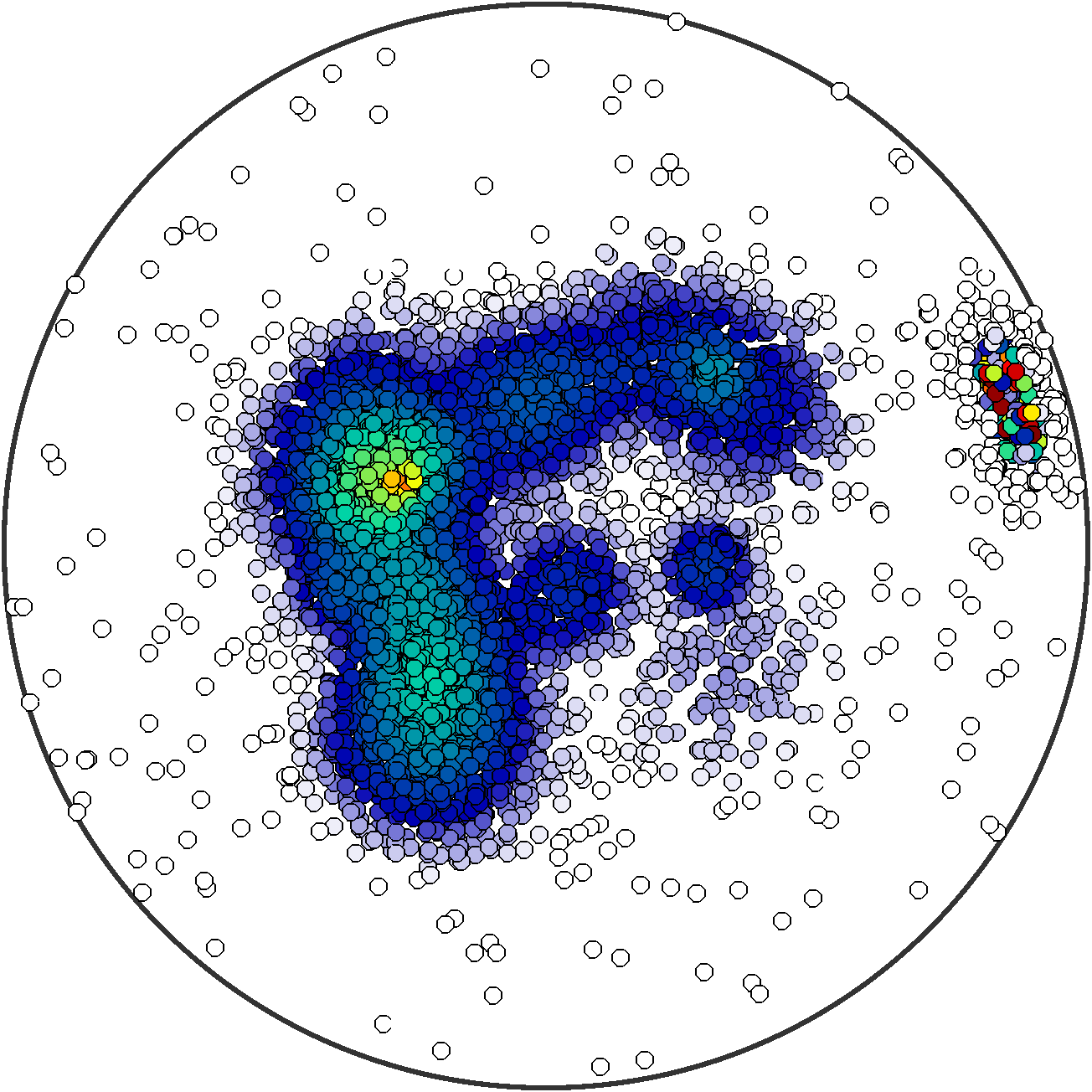}
      \subcaption{The $N=10^{4}$ scattered nodes}\label{fig:S2ToyNodes}
    \end{subfigure}
  \end{minipage}
  \hspace{0.02\textwidth}
  \begin{minipage}[c]{0.10\textwidth}
    \centering
    \includegraphics[height=\toyPanelSize, trim={330bp 0bp 5bp 0bp},clip]{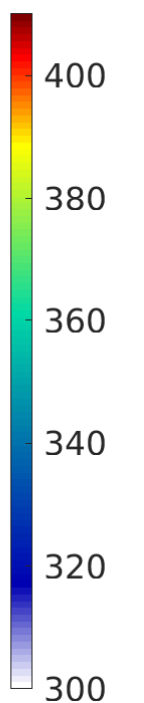}
  \end{minipage}
  \caption{Stereographic projection of the upper hemisphere, showing the
    synthetic test function $f_{0}$ of bandwidth $512$ and the $N=10^{4}$ nodes at which it is sampled.
    The nodes are drawn from a smoothed copy of $f_{0}-300$, up to a uniform fraction of $4\,\%$, and are therefore ill-distributed with fill distance $h_X\approx14.6^\circ$.
    Both panels share the color scale on the right.}\label{fig:S2ToyFunctionData}
\end{figure}

Now, we compute the approximation $\tilde f\in\bandlimFunc{L}{\S{2}}$ from the scattered data $\{(x_{i},f_{0}(x_{i}))\}_{i=1}^{N}$ and measure the relative error
\begin{equation} \label{eq:relative_error_S2}
  E(\tilde f) \coloneqq \frac{\norm{\tilde f-f_{0}}_{\L2(\S{2})}}{\norm{f_{0}}_{\L2(\S{2})}},
\end{equation}
as a function of the reconstruction bandwidth $L$, as illustrated in \Cref{fig:S2TestFunction_TimeError}.

For LSQR, the regularization parameter is chosen optimally for each bandwidth
separately by minimizing $E$.
This heuristic is computationally expensive and, of course, favors the LSQR
approximation error. Moreover, this approach is only possible because the exact
solution is known. Both methods use the Voronoi areas of the nodes: as
quadrature weights in LSQR and within the local MLS weights in HAMLS.

\begin{figure}[htbp]
  \centering
  \begin{tikzpicture}
      \begin{axis}[
        width = 0.7\textwidth,
        height = 0.42\textwidth,
        xmin=4, xmax=128,
        ymin=0.00035, ymax=0.055,
        ymode=log, log basis y={2},
        xtick={16,32,48,64,80,96,112,128},
        ytick={0.0005,0.001,0.002,0.004,0.008,0.016,0.032},
        yticklabels={0.0005,0.001,0.002,0.004,0.008,0.016,0.032},
        minor tick style = {draw=none},
        ymajorgrids,
        xmajorgrids,
        x grid style = solid,
        font = {\scriptsize},
        legend style={text width=2.8cm, align=left,
          fill opacity=1, draw opacity=1, text opacity=1},
        legend pos=south west,
        xlabel={Bandwidth $L$},
        ylabel={Relative error $E(\tilde f)$},
        ]

        \addplot[color=black!55, line width=1.2pt, each nth point=2] table [x=bandwidth,
        y=Projection, col sep=space]
        {code/synthetic/results/errorVsBandwidth/ApproximationError.txt};
        \addlegendentry{$\L2$-projection}

        \addplot[color=blue, solid, line width=1pt, each nth point=2] table [x=bandwidth, y=LSQR,
        col sep=space] {code/synthetic/results/errorVsBandwidth/ApproximationError.txt};
        \addlegendentry{LSQR}

        \addplot[color=Green, dotted, line width=1.2pt, each nth point=2] table [x=bandwidth,
        y=MLS_Harm_deg2, col sep=space]
        {code/synthetic/results/errorVsBandwidth/ApproximationError.txt};
        \addlegendentry{HAMLS with degree 2}

        \addplot[color=orange, dashdotted, thick, each nth point=2] table [x=bandwidth,
        y=MLS_Harm_deg3, col sep=space]
        {code/synthetic/results/errorVsBandwidth/ApproximationError.txt};
        \addlegendentry{HAMLS with degree 3}

        \addplot[color=red, dashed, thick, each nth point=2] table [x=bandwidth,
        y=MLS_Harm_deg4, col sep=space]
        {code/synthetic/results/errorVsBandwidth/ApproximationError.txt};
        \addlegendentry{HAMLS with degree 4}
      \end{axis}
  \end{tikzpicture}
  \caption{Relative $\L2(\S{2})$-error~\eqref{eq:relative_error_S2} of the
    harmonic approximations for varying bandwidth. HAMLS (non-solid) is computed
    with target neighbor count
      $n = 4 \cdot \dim(\Pi_K^{\mathrm{tan}}(\S{2};x))$ and varying polynomial
    degrees $K$. The LSQR approximation (solid blue) of the regularized
    problem~\eqref{eq:LSQR-Problem_Regularized} uses Sobolev index $2$ and an
    optimal regularization parameter and is iterated until its stopping
    criterion is met, yielding a good approximation to the minimizer of
    \eqref{eq:LSQR-Problem_Regularized}. 
    Using a stricter stopping tolerance, and hence a larger time budget,
      the LSQR error decreases slightly further, cf.~\Cref{fig:Error_x*Time}. 
    The solid gray curve shows the truncation error of the $\L2$-orthogonal
    projection of $f_{0}$ onto $\bandlimFunc{L}{\S{2}}$ and serves as a lower
    error bound.}\label{fig:S2TestFunction_TimeError} 
\end{figure}
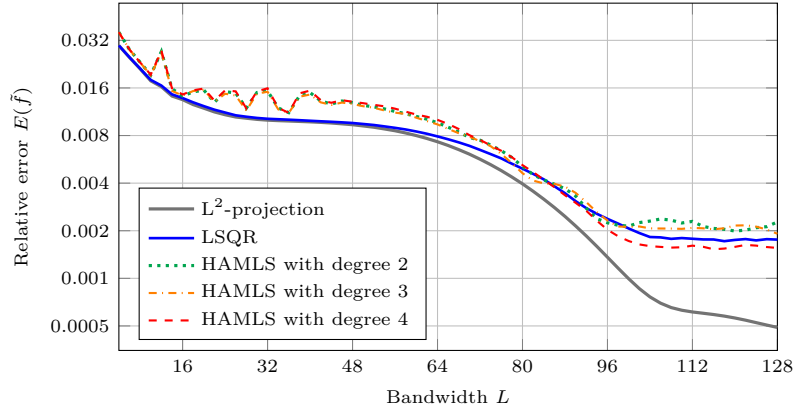

At bandwidth $L=128$, the HAMLS algorithm requires only $0.7$--$2.1$ seconds.
In contrast, LSQR requires about $70$ seconds and $6908$ iterations to obtain a
similarly accurate approximation. However, both methods yield errors notably
above the lower bound set by the best approximation in $\bandlimFunc{L}{\S{2}}$,
\ie the $\L2(\S{2})$-orthogonal projection $P_L f_0$. 

Furthermore, we observe that for $L\gtrsim100$, the HAMLS errors stagnate at
around $2\cdot10^{-3}$, although the $\L2$-projection error continues to decay.
The onset of stagnation coincides with $(L+1)^{2}>N$, at which point the
  available data, and no longer the bandwidth, limit the accuracy. This
  corresponds to the $L$-independent term in
  \Cref{thm:hamls_error_bound_good_data}. Although the assumptions of this
  theorem are not met here, between $44\,\%$ and $89\,\%$ of the sphere is
  reconstructed without regularization, depending on the polynomial degree, so
  that this term still enters the $\L2$-error through averaging.

We further observe that increasing the polynomial degree from $2$ to $3$
provides no significant improvement in the HAMLS approximation. This behavior is
explained in the following remark.

\begin{samepage}
\begin{remark}
  In \Cref{fig:S2TestFunction_TimeError}, the achieved relative errors for HAMLS
  with MLS degree $K=2$ and $K=3$ are very similar.
  This is not a coincidence.
  Superconvergence one order higher than expected has already been observed and
  explained in \cite{Li2020} for MLS using even polynomial degrees and is
  currently being revisited for regular grids \cite{Durst2026}. It is usually expected
  for uniformly distributed nodes and unregularized MLS. However, the
  $\L2$-error measured here is a form of averaging the errors of MLS over many
  points, and, depending on the degree of MLS, $44\,\%$ to $89\,\%$ of the sphere
  does not receive regularization. This allows us to see this behavior even for
  ill-distributed nodes.   
\end{remark}
\end{samepage}

To fairly compare both methods, we first need to determine an appropriate runtime budget for LSQR.
We initially attempted to use the same runtime budget for both methods, but
LSQR does not provide a meaningful approximation within the runtime required
by HAMLS unless the bandwidth is small.
We therefore examine how the LSQR approximation develops as the available runtime increases.
Specifically, we allow LSQR a runtime of $1,5,10,\ldots,100$ times the corresponding runtime of HAMLS with polynomial degree $4$.
For each runtime budget, the regularization parameter is chosen to minimize the approximation error.
The resulting errors are shown in \Cref{fig:Error_x*Time}.
We see that, given a very large time budget, the LSQR error indeed becomes smaller than the error of HAMLS with the chosen parameters.
However, choosing a higher polynomial degree in HAMLS would yield a better approximation as well.

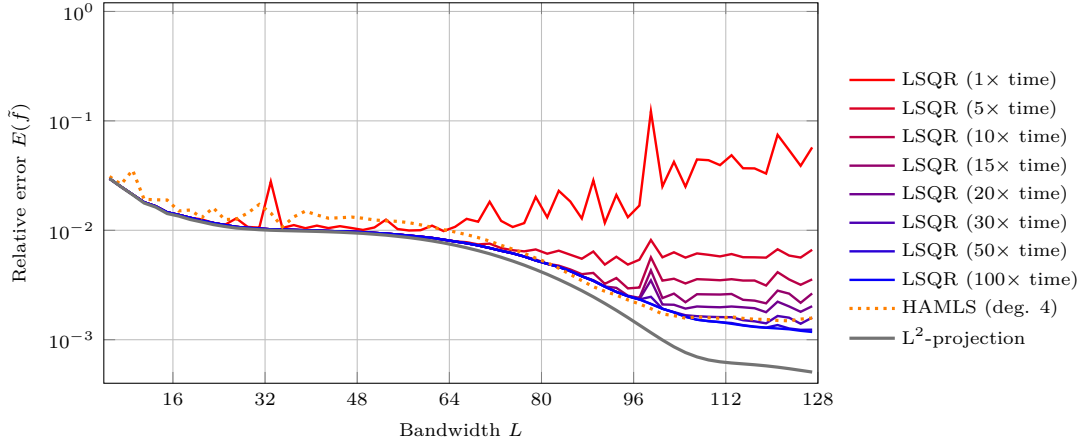
\begin{figure}[htbp]
  \centering
  \begin{tikzpicture}
      \begin{axis}[
        name=leftaxis,
        width = 0.75\textwidth,
        height = 0.45\textwidth,
        xmin=4, xmax=128,
        ymin=0.0004, ymax=1.2,
        ymode=log, log basis y={10},
        xtick={16,32,48,64,80,96,112,128},
        minor tick style = {draw=none},
        ymajorgrids,
        xmajorgrids,
        every axis plot/.append style={x filter/.expression={mod(x,2)==1 ? x : nan}},
        font = {\scriptsize},
        legend style={at={(1.03,0.85)}, anchor=north west, draw=none,
          cells={anchor=west}, font=\scriptsize},
        legend cell align=left,
        xlabel={Bandwidth $L$},
        ylabel={Relative error $E(\tilde f)$},
        ]

        \addplot[color=blue!0!red, solid, line width=0.9pt] table [x=bandwidth, y=LSQR_x1, col sep=space] {code/synthetic/results/equalTime/LSQRErrorTimeFactors.txt};
        \addlegendentry{LSQR ($1\times$ time)}

        \addplot[color=blue!14!red, solid, line width=0.9pt] table [x=bandwidth, y=LSQR_x5, col sep=space] {code/synthetic/results/equalTime/LSQRErrorTimeFactors.txt};
        \addlegendentry{LSQR ($5\times$ time)}

        \addplot[color=blue!28!red, solid, line width=0.9pt] table [x=bandwidth, y=LSQR_x10, col sep=space] {code/synthetic/results/equalTime/LSQRErrorTimeFactors.txt};
        \addlegendentry{LSQR ($10\times$ time)}

        \addplot[color=blue!42!red, solid, line width=0.9pt] table [x=bandwidth, y=LSQR_x15, col sep=space] {code/synthetic/results/equalTime/LSQRErrorTimeFactors.txt};
        \addlegendentry{LSQR ($15\times$ time)}

        \addplot[color=blue!56!red, solid, line width=0.9pt] table [x=bandwidth, y=LSQR_x20, col sep=space] {code/synthetic/results/equalTime/LSQRErrorTimeFactors.txt};
        \addlegendentry{LSQR ($20\times$ time)}

        \addplot[color=blue!70!red, solid, line width=0.9pt] table [x=bandwidth, y=LSQR_x30, col sep=space] {code/synthetic/results/equalTime/LSQRErrorTimeFactors.txt};
        \addlegendentry{LSQR ($30\times$ time)}

        \addplot[color=blue!84!red, solid, line width=0.9pt] table [x=bandwidth, y=LSQR_x50, col sep=space] {code/synthetic/results/equalTime/LSQRErrorTimeFactors.txt};
        \addlegendentry{LSQR ($50\times$ time)}

        \addplot[color=blue!98!red, solid, line width=0.9pt] table [x=bandwidth, y=LSQR_x100, col sep=space] {code/synthetic/results/equalTime/LSQRErrorTimeFactors.txt};
        \addlegendentry{LSQR ($100\times$ time)}

        \addplot[color=orange, dotted, line width=1.2pt] table [x=bandwidth, y=MLS_Harm_deg4, col sep=space]
        {code/synthetic/results/equalTime/LSQRErrorTimeFactors.txt};
        \addlegendentry{HAMLS (deg. 4)}

        \addplot[color=black!55, line width=1.2pt] table [x=bandwidth, y=Projection, col sep=space]
        {code/synthetic/results/equalTime/LSQRErrorTimeFactors.txt};
        \addlegendentry{$\L2$-projection}

      \end{axis}
  \end{tikzpicture}
  \caption{Relative $\L2(\S{2})$-error~\eqref{eq:relative_error_S2} of the LSQR
    method when granted $1,5,10,\dots,100$ times the runtime of the HAMLS
    approximation (polynomial degree $4$ and target neighbor count $n = 60$), at
    the same bandwidth and with the error-minimizing regularization parameter.
    The runtime budget of each curve determines the number of LSQR iterations
    that fit into the corresponding multiple of the HAMLS runtime. Only odd
    bandwidths are shown: on the machine used, one LSQR iteration costs up to
    twice as much for $L \equiv 4 \ (\mathrm{mod}\ 8)$ as for the neighboring
    bandwidths, so the number of iterations that fit into a budget, and with it
    the error, would oscillate strongly between neighboring bandwidths. The
    remaining spikes have the same cause. The gray curve shows the truncation
    error of the $\L2$-orthogonal projection of $f_{0}$ onto
    $\bandlimFunc{L}{\S{2}}$, which provides a lower bound for every
    band-limited approximation.}\label{fig:Error_x*Time} 
\end{figure}

Finally, as already mentioned in \Cref{sec:quadrature grid}, we compare the
runtimes of steps 2 and 3, namely the MLS evaluation on the quadrature grid and
the quadrature by the adjoint $\S2$-Fourier transform, respectively,
see~\Cref{tab:TimeHAMLSsteps}. The table shows that the MLS evaluation in step 2
is the main contributor to the overall runtime of HAMLS.

\begin{table}[htbp]
  \centering
  \caption{Runtime in seconds of the MLS evaluation (step~2 of HAMLS) for $10^4$, $10^5$ and $10^6$ scattered data points and different polynomial degrees, compared with the runtime of the adjoint spherical Fourier transform used in quadrature (step~3 of HAMLS).}\label{tab:TimeHAMLSsteps}
  \small
  \setlength{\tabcolsep}{9pt}
  \setlength{\aboverulesep}{0pt}
  \setlength{\belowrulesep}{0pt}
  \setlength{\extrarowheight}{2pt}

  \pgfplotstableread[col sep=space]{code/synthetic/results/stepRuntimes/RuntimeMLSQuadrature.txt}\runtimetable

  \pgfplotstabletypeset[
    row predicate/.code={%
      \pgfplotstablegetelem{#1}{bandwidth}\of\runtimetable
      \pgfmathparse{Mod(\pgfplotsretval,16)==0 && \pgfplotsretval>0}%
      \ifdim\pgfmathresult pt<0.5pt \pgfplotstableuserowfalse\fi
    },
    columns={bandwidth,Quadrature,
      MLS_1e4_deg2,MLS_1e5_deg2,MLS_1e6_deg2,
      MLS_1e4_deg4,MLS_1e5_deg4,MLS_1e6_deg4},
    every head row/.style={
      output empty row,
      before row={%
        \toprule
          \multirow{2}{*}{$L$} & \multirow{2}{*}{Quadrature}
          & \multicolumn{3}{c|}{MLS with polynomial degree 2}
          & \multicolumn{3}{c|}{MLS with polynomial degree 4}\\
        \cline{3-8}
          &
          & $N=10^4$ & $N=10^5$ & $N=10^6$
          & $N=10^4$ & $N=10^5$ & $N=10^6$\\
        \midrule},
    },
    every last row/.style={after row=\bottomrule},
    every column/.style={column type=c, fixed, fixed zerofill, precision=3},
    columns/bandwidth/.style={column type={|c|}, int detect},
    columns/Quadrature/.style={column type=c|},
    columns/MLS_1e6_deg2/.style={column type={c|}},
    columns/MLS_1e6_deg4/.style={column type={c|}},
  ]\runtimetable

\end{table}

We see that the MLS evaluation is almost independent of the data size, as $N$
enters only through the nearest-neighbor queries. More precisely, increasing $N$
from $10^4$ to $10^6$ has no significant impact on the running time for
intermediate to large bandwidths.

The cost of HAMLS is therefore governed primarily by the bandwidth and, consequently, by the number of quadrature points $M_Q$.
The dependence of the local problems in \Cref{sec:mls transfer} on the polynomial degree is much milder than the size of these problems might suggest.
Increasing the polynomial degree of MLS from $2$ to $4$ increases $\dim(\Pi_K^{\mathrm{tan}}(\S{2};x))$ from $6$ to $15$, so that the cost of each local least-squares solve grows superlinearly with the problem dimension.
Nevertheless, the measured runtime only increases by a factor that mostly stays
below $3.9$.

\section{Conclusion}
\label{sec:conclusion}

The HAMLS algorithm obtains a global harmonic expansion by combining MLS with a
fast Fourier transform on a quadrature grid. In the MLS step, adapting the
support radius and the regularization to the local node geometry is crucial for 
handling ill-distributed data. At the same time, the global step remains stable
and non-iterative. For exact samples of sufficiently smooth functions on
sufficiently dense quasi-uniform nodes, the $\L2$-error with unregularized MLS
is bounded by the error of the harmonic approximation arising from exact grid
values, plus a term of order $h_X^{K+1}$ independent of the bandwidth. Here $K$
is the polynomial degree employed in the MLS step. The experiments on $\S{2}$
demonstrate the practical benefit: HAMLS achieves errors comparable to those of
LSQR at a fraction of the runtime, with the advantage increasing with bandwidth.

\printbibliography

\end{document}